\documentclass[12pt]{article}

\usepackage{amssymb}
\usepackage{amsthm}

\usepackage{pict2e}

\newtheorem{prop}{Proposition}[section]
\newtheorem{cor}[prop]{Corollary}
\newtheorem{thm}[prop]{Theorem}
\newtheorem{lem}[prop]{Lemma}
\theoremstyle{definition}
\newtheorem{definition}[prop]{Definition}
\newtheorem{example}[prop]{Example}
\newtheorem{rem}[prop]{Remark}

\def\Core{{\rm Core}}
\def\Irr{{\rm Irr}}

\def\GL{\mathrm{GL}}
\def\PSL{\mathrm{PSL}}
\def\PGL{\mathrm{PGL}}
\def\PSU{\mathrm{PSU}}
\def\PSp{\mathrm{PSp}}
\def\POmega{\mathrm{P\Omega}}
\def\SL{\mathrm{SL}}
\def\Sp{\mathrm{Sp}}
\def\SU{\mathrm{SU}}
\def\Sz{\mathrm{Sz}}
\def\AGL{\mathrm{AGL}}
\def\spl{\!:\!}

\begin{document}

\title{On Frobenius graphs of diameter $3$ \\
for finite groups}
\author{T. Breuer, L. H\'ethelyi, E. Horv\'ath and B. K\"ulshammer}

\date{July 10th, 2024}

\maketitle

\abstract{%
For a subgroup $H$ of a finite group $G$,
the Frobenius graph $\Gamma(G, H)$ records the constituents
of the restrictions to $H$ of the irreducible characters of $G$.
We investigate when this graph has diameter $3$.}

\bigskip
\noindent
\emph{In memory of our friend and colleague Erzs\'ebet Horv\'ath,
who sadly passed away during the preparation of this paper.}

\section{Introduction}\label{sect:intro}


Let $H$ be a subgroup of a finite group $G$.
The \emph{Frobenius graph} $\Gamma(G, H)$ is the bipartite graph
with vertex set the disjoint union of $\Irr(G)$ and $\Irr(H)$
and an edge between $\chi \in \Irr(G)$ and $\varphi \in \Irr(H)$
whenever $[\chi_H, \varphi] \neq 0$.
Here $[\alpha, \beta] = [\alpha, \beta]_H$ is the inner product
of (virtual) characters $\alpha, \beta$ of $H$,
and $\chi_H$ denotes the restriction of a (virtual) character $\chi$ of $G$
to $H$.
We also define the \emph{Frobenius matrix}
\[
   F(G, H):= ([\chi_H, \varphi])_{\varphi \in \Irr(H), \chi \in \Irr(G)}.
\]
Then the biadjacency matrix of $\Gamma(G, H)$ is obtained from $F(G, H)$
by replacing each nonzero entry by $1$.

We will mainly be interested in situations where
the \emph{diameter} $d$ of $\Gamma(G, H)$,
i.~e., the largest distance of two vertices in $\Gamma(G, H)$, is finite.
We show that $\Gamma(G, H)$ is connected if and only if the
core of $H$ in $G$ is trivial.
We note:

\begin{itemize}
\item
    $d = 1$ if and only if $|G| = 1$ holds,
\item
    $d = 2$ if and only if $|G| > 1$ and $|H| = 1$ hold.
\end{itemize}

Thus $d \geq 3$ holds if and only if $H$ is nontrivial.
Observe that for $1_H \neq \varphi \in \Irr(H)$,
any path from $1_G$ to $\varphi$ in $\Gamma(G, H)$ has odd length $> 1$.

Frobenius graphs can have arbitrarily large diameters,
for example the diameter of $\Gamma(S_{n+1}, S_n)$ is $2n$,
where $S_n$ is the symmetric group of degree $n$.
This can be proved in terms of partitions of $n$,
which parametrize the irreducible characters of $S_n$, as follows:
Any partition of $n$ can be transformed into any other partition of $n$
by a sequence of at most $n-1$ steps, where one step consists of adding an
addable node in order to get a partition of $n+1$,
and then removing a removable node.
Thus one gets a path of length at most $2n-2$ between any two characters
of $S_n$ in $\Gamma(S_{n+1}, S_n)$.
Then it is clear that paths between characters of $S_{n+1}$
have length at most $2n$, and that the path between the trivial and the
alternating character of $S_{n+1}$ has length exactly $2n$.

In the following,
we will investigate Frobenius graphs $\Gamma(G, H)$ of diameter $3$.
In this case we call $H$ a \emph{diameter three subgroup} of $G$.


Section~\ref{sect:struct} collects properties of pairs $(G, H)$
such that $H$ is a diameter three subgroup of $G$,
Section~\ref{sect:examples} shows examples,
Section~\ref{sect:large} studies large subgroups that are
diameter three subgroups,
Section~\ref{sect:quasisimple} classifies those quasisimple groups that
contain diameter three subgroups,
Section~\ref{sect:minimal} deals with the question how minimal groups
(w.~r.~t.~inclusion) look like which have a diameter three subgroup.
Finally, Section~\ref{sect:depth} explains some connections between
the diameter of a Frobenius graph $\Gamma(G, H)$
and the subgroup depth of $H$ in $G$.

The direct computations for this paper were done with the help of the
computer algebra system Oscar~\cite{OSCAR}.

\section{Structural Properties}\label{sect:struct}

The following result is essentially contained in~\cite[Section~6]{BKK}.
We include a proof for the convenience of the reader.

\begin{prop}\label{corefree}
Let $H$ be a proper subgroup of a finite group $G$.
Then the connected components of $\Gamma(G, H)$ are in bijection with
the $G$-orbits on $\Irr(K)$
where $K := \Core_G(H) := \bigcap_{g \in G} gHg^{-1}$ denotes
the core of $H$ in $G$.
In particular, $\Gamma(G, H)$ is connected if and only if $K = 1$.
\end{prop}

\begin{proof}
Let $\chi \in \Irr(G)$ and $\varphi \in \Irr(H)$ such that
$[\chi_H, \varphi] \neq 0$.
By Clifford theory, $\varphi_K$ is a sum of $H$-conjugates of a character
$\theta \in \Irr(K)$.
Thus $\chi_K$ is a sum of $G$-conjugates of $\theta$.
If also $\psi \in \Irr(G)$ satisfies $[\psi_H, \varphi] \neq 0$ then
$\psi_K$ is again a sum of $G$-conjugates of $\theta$.
Thus, whenever $\chi' \in \Irr(G)$ is contained in the same
connected component of $\Gamma(G, H)$ as $\chi$ then the
irreducible constituents of $\chi_K'$ form the $G$-orbit of $\theta$.

Conversely, let $\chi, \chi' \in \Irr(G)$ such that
$0 \neq [\chi_K, \chi_K'] = [1_K, \overline{\chi}_K\chi_K']
                          = [1_K^G, \overline{\chi}\chi']$.
Then $\overline{\chi}\chi'$ has a constituent $\eta$
whose kernel contains $K$.
On the other hand, the kernel of $1_H^G$ is $K$,
so that we can view $1_H^G$ as a faithful character of $G/K$.
By a theorem of Burnside~\cite[Satz~V.10.8]{HuppertI},
there is a positive integer $m$ such that
$0 < [(1_H^G)^m, \eta] \leq [(1_H^G)^m, \overline{\chi} \chi']
   = [\chi (1_H^G)^m, \chi']$.
Consider the linear map
$U: \mathbb{Z} \Irr(G) \longrightarrow \mathbb{Z} \Irr(G)$,
$\alpha \longmapsto \alpha_H^G = \alpha (1_H^G)$;
here $\mathbb{Z} \Irr(G)$ denotes the group of virtual characters of $G$.
Then $U^n(\alpha) = \alpha (1_H^G)^n$ for all $n>0$
which implies that $0 \neq [U^m(\chi), \chi']$.
Now note that the constituents of $U(\chi) = \chi_H^G$ are contained in
the connected component of $\chi$ in $\Gamma(G, H)$.
Thus also the constituents of $U^m(\chi)$ are contained in the
connected component of $\chi$ in $\Gamma(G, H)$.
In particular, $\chi'$ is contained in the connected component of $\chi$
in $\Gamma(G, H)$.

This proves the first assertion of our proposition.
Now suppose that $\Gamma(G, H)$ is connected.
Then $\Irr(K)$ is a single $G$-orbit, i.~e.,
$G$ acts transitively on $\Irr(K)$.
Thus the trivial character $1_K$ is the only irreducible character of $K$,
so that $K = 1$.
The converse is clear.
\end{proof}

By Proposition~\ref{corefree} and the remarks from Section~\ref{sect:intro},
a diameter three subgroup $H$ is always nontrivial
and \emph{core-free} in $G$, i.~e., $\Core_G(H) = 1$.
The following result will be our main tool in order to detect
diameter three subgroups.

\begin{prop}\label{diam3_equiv}
Let $H$ be a nontrivial proper subgroup of a finite group $G$.
Then the following assertions are equivalent:

\begin{itemize}
\item[(a)]
    $\Gamma(G, H)$ has diameter $3$.
\item[(b)]
    \begin{itemize}
    \item[(i)]
        For any $\chi \in \Irr(G)$, we have $[\chi_H, 1_H] \neq 0$.
    \item[(ii)]
        For any $\varphi, \psi \in \Irr(H)$,
        we have $[\varphi^G, \psi^G] \neq 0$.
  \end{itemize}
\end{itemize}
\end{prop}

\begin{proof}
Suppose that~(a) holds.
Then any $\chi \in \Irr(G)$ is connected to the trivial character $1_G$ of $G$
via a path of length $l \leq 3$ in $\Gamma(G, H)$.
Since $l$ must be even, we have in fact $l \leq 2$.
Thus $\chi_H$ and $(1_G)_H = 1_H$ have a common constituent,
which must be $1_H$. Hence~(i) holds.

Similarly,
any two $\varphi, \psi \in \Irr(H)$ are connected
via a path of length $l \leq 3$.
Again, we must have $l \leq 2$.
Thus there is $\chi \in \Irr(G)$ such that
$[\chi_H, \varphi] \neq 0 \neq [\chi_H, \psi]$ holds.
This implies that $[\varphi^G, \psi^G] \neq 0$, and~(ii) holds.

Now suppose that~(b) holds.
Then~(i) implies that $\Gamma(G, H)$ is connected.
More precisely,
the distance in $\Gamma(G, H)$ between any two characters
$\chi, \eta \in \Irr(G)$ is at most $2$.
Similarly, by~(ii) the distance between any two characters
$\varphi, \psi \in \Irr(H)$ is at most $2$.
Thus the distance between any character $\chi \in \Irr(G)$
and any character $\varphi \in \Irr(H)$ is at most $3$.
Since $H$ is nontrivial, the diameter of $\Gamma(G, H)$ is $3$.
\end{proof}

\begin{rem}\label{independent}
We note that the conditions (b)~(i) and (b)~(ii)
in Proposition~\ref{diam3_equiv} are independent.

In many examples, condition~(b)~(ii) does not imply condition~(b)~(i).
We can take the symmetric groups $G = S_3$ and $H = S_2$ of degrees $3$
and $2$, where $\Gamma(G, H)$ is a path of length $4$.

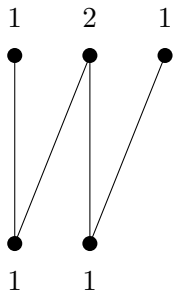
\setlength{\unitlength}{1mm}

\begin{center}
\begin{picture}(20,35)
\put(0,5){\circle*{2}}
\put(10,5){\circle*{2}}
\put(0,0){\makebox(0,0){$1$}}
\put(10,0){\makebox(0,0){$1$}}
\put(0,30){\circle*{2}}
\put(10,30){\circle*{2}}
\put(20,30){\circle*{2}}
\put(0,35){\makebox(0,0){$1$}}
\put(10,35){\makebox(0,0){$2$}}
\put(20,35){\makebox(0,0){$1$}}
\put(0,5){\line(0,500){25}}
\put(0,5){\line(200,500){10}}
\put(10,5){\line(0,500){25}}
\put(10,5){\line(200,500){10}}
\end{picture}
\end{center}

An example where~(b)~(i) does not imply~(b)~(ii)
is given by the Frobenius group $G$ of order $351 = 3^3 \cdot 13$
with an elementary abelian kernel of order $3^3$
and a complement of order $13$,
and a subgroup $H$ of order $3^2$ in $G$.
Then $G$ is a subgroup of index 2 in $\AGL(1, 3^3)$.
Moreover, $\Irr(G)$ consists of $13$ linear characters and two irreducible
characters $\chi, \eta$ of degree $13$.
Furthermore, condition~(b)~(i) is satisfied
since the permutation character $1_H^G$ is exactly the sum of all
irreducible characters of $G$.
On the other hand, condition~(b)~(ii) is not satisfied
since there are characters $\varphi, \psi \in \Irr(H)$ such that
$\varphi^G = 3 \chi$ and $\psi^G = 3 \eta$, so that $[\varphi^G, \psi^G] = 0$.

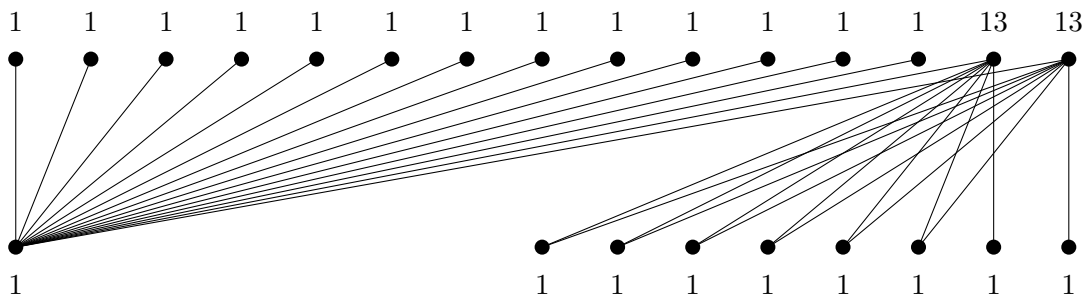
\setlength{\unitlength}{1mm}

\begin{center}
\begin{picture}(140,35)
\put(0,5){\circle*{2}}
\put(70,5){\circle*{2}}
\put(80,5){\circle*{2}}
\put(90,5){\circle*{2}}
\put(100,5){\circle*{2}}
\put(110,5){\circle*{2}}
\put(120,5){\circle*{2}}
\put(130,5){\circle*{2}}
\put(140,5){\circle*{2}}
\put(0,0){\makebox(0,0){$1$}}
\put(70,0){\makebox(0,0){$1$}}
\put(80,0){\makebox(0,0){$1$}}
\put(90,0){\makebox(0,0){$1$}}
\put(100,0){\makebox(0,0){$1$}}
\put(110,0){\makebox(0,0){$1$}}
\put(120,0){\makebox(0,0){$1$}}
\put(130,0){\makebox(0,0){$1$}}
\put(140,0){\makebox(0,0){$1$}}
\put(0,30){\circle*{2}}
\put(10,30){\circle*{2}}
\put(20,30){\circle*{2}}
\put(30,30){\circle*{2}}
\put(40,30){\circle*{2}}
\put(50,30){\circle*{2}}
\put(60,30){\circle*{2}}
\put(70,30){\circle*{2}}
\put(80,30){\circle*{2}}
\put(90,30){\circle*{2}}
\put(100,30){\circle*{2}}
\put(110,30){\circle*{2}}
\put(120,30){\circle*{2}}
\put(130,30){\circle*{2}}
\put(140,30){\circle*{2}}
\put(0,35){\makebox(0,0){$1$}}
\put(10,35){\makebox(0,0){$1$}}
\put(20,35){\makebox(0,0){$1$}}
\put(30,35){\makebox(0,0){$1$}}
\put(40,35){\makebox(0,0){$1$}}
\put(50,35){\makebox(0,0){$1$}}
\put(60,35){\makebox(0,0){$1$}}
\put(70,35){\makebox(0,0){$1$}}
\put(80,35){\makebox(0,0){$1$}}
\put(90,35){\makebox(0,0){$1$}}
\put(100,35){\makebox(0,0){$1$}}
\put(110,35){\makebox(0,0){$1$}}
\put(120,35){\makebox(0,0){$1$}}
\put(130,35){\makebox(0,0){$13$}}
\put(140,35){\makebox(0,0){$13$}}
\put(0,5){\line(0,500){25}}
\put(0,5){\line(200,500){10}}
\put(0,5){\line(400,500){20}}
\put(0,5){\line(600,500){30}}
\put(0,5){\line(800,500){40}}
\put(0,5){\line(1000,500){50}}
\put(0,5){\line(1200,500){60}}
\put(0,5){\line(1400,500){70}}
\put(0,5){\line(1600,500){80}}
\put(0,5){\line(1800,500){90}}
\put(0,5){\line(2000,500){100}}
\put(0,5){\line(2200,500){110}}
\put(0,5){\line(2400,500){120}}
\put(0,5){\line(2600,500){130}}
\put(0,5){\line(2800,500){140}}
\put(70,5){\line(1200,500){60}}
\put(70,5){\line(1400,500){70}}
\put(80,5){\line(1000,500){50}}
\put(80,5){\line(1200,500){60}}
\put(90,5){\line(800,500){40}}
\put(90,5){\line(1000,500){50}}
\put(100,5){\line(600,500){30}}
\put(100,5){\line(800,500){40}}
\put(110,5){\line(400,500){20}}
\put(110,5){\line(600,500){30}}
\put(120,5){\line(200,500){10}}
\put(120,5){\line(400,500){20}}
\put(130,5){\line(0,500){25}}
\put(140,5){\line(0,500){25}}
\end{picture}
\end{center}

Examples of groups for which~(b)~(i) does not imply~(b)~(ii) for some
subgroup seem to be rare,
see Remark~\ref{rem:smallpermdegree} and Section~\ref{subsect:minsmall}.
\end{rem}

\begin{rem}
In the situation of Proposition~\ref{diam3_equiv},
for any two characters $\varphi,\psi \in \Irr(H)$
the induced characters $\varphi^G, \psi^G$ have a common constituent
$\chi \in \Irr(G)$.
However, in general there does not exist a character $\chi \in \Irr(G)$
which is a constituent of $\varphi^G$ for \emph{every} $\varphi \in \Irr(H)$.
As an example, one can take the Frobenius group $G$ of order $2^4 \cdot 5$
and a suitable subgroup $H$ of order $4$.
(There are seven classes of subgroups of order four, one of them works.)
Then the Frobenius matrix $F(G, H)$ and the Frobenius graph $\Gamma(G, H)$
are as follows:

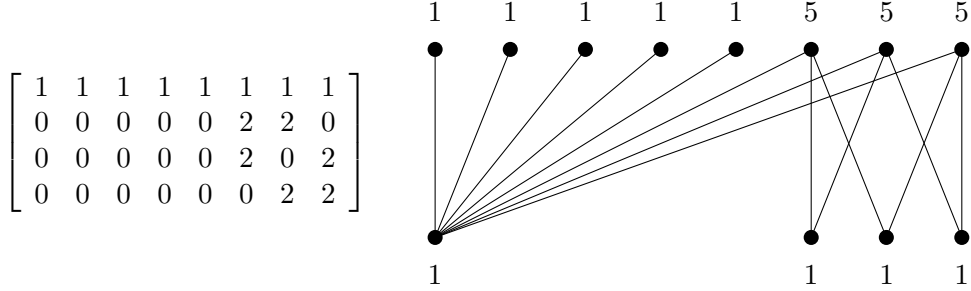
\setlength{\unitlength}{1mm}

\begin{minipage}{4.5cm}
  \vspace*{1.5cm}
  $\displaystyle\left[
    \begin{array}{cccccccc}
         1 & 1 & 1 & 1 & 1 & 1 & 1 & 1 \\
         0 & 0 & 0 & 0 & 0 & 2 & 2 & 0 \\
         0 & 0 & 0 & 0 & 0 & 2 & 0 & 2 \\
         0 & 0 & 0 & 0 & 0 & 0 & 2 & 2
    \end{array}
  \right]$
  \vspace*{1.5cm}
\end{minipage} \hspace*{1cm} \begin{minipage}{6cm}
\begin{picture}(70,35) 
\put(0,5){\circle*{2}}
\put(50,5){\circle*{2}}
\put(60,5){\circle*{2}}
\put(70,5){\circle*{2}}
\put(0,0){\makebox(0,0){$1$}}
\put(50,0){\makebox(0,0){$1$}}
\put(60,0){\makebox(0,0){$1$}}
\put(70,0){\makebox(0,0){$1$}}
\put(0,30){\circle*{2}}
\put(10,30){\circle*{2}}
\put(20,30){\circle*{2}}
\put(30,30){\circle*{2}}
\put(40,30){\circle*{2}}
\put(50,30){\circle*{2}}
\put(60,30){\circle*{2}}
\put(70,30){\circle*{2}}
\put(0,35){\makebox(0,0){$1$}}
\put(10,35){\makebox(0,0){$1$}}
\put(20,35){\makebox(0,0){$1$}}
\put(30,35){\makebox(0,0){$1$}}
\put(40,35){\makebox(0,0){$1$}}
\put(50,35){\makebox(0,0){$5$}}
\put(60,35){\makebox(0,0){$5$}}
\put(70,35){\makebox(0,0){$5$}}
\put(0,5){\line(0,500){25}}
\put(0,5){\line(200,500){10}}
\put(0,5){\line(400,500){20}}
\put(0,5){\line(600,500){30}}
\put(0,5){\line(800,500){40}}
\put(0,5){\line(1000,500){50}}
\put(0,5){\line(1200,500){60}}
\put(0,5){\line(1400,500){70}}
\put(50,5){\line(0,500){25}}
\put(50,5){\line(200,500){10}}
\put(60,5){\line(-200,500){10}}
\put(60,5){\line(200,500){10}}
\put(70,5){\line(-200,500){10}}
\put(70,5){\line(0,500){25}}
\end{picture}
\end{minipage}

Thus $\Gamma(G, H)$ has diameter $3$.
\end{rem}

\begin{rem}\label{reformulate}
Condition~(b)~(ii) of Proposition~\ref{diam3_equiv} is equivalent
to the following one.
For every $\varphi, \psi \in \Irr(H)$, there exists $g \in G$
such that $[\varphi_{H \cap H^g}, \psi^g_{H \cap H^g}] > 0$.

Note that by Frobenius reciprocity and Mackey decomposition,
\[
   [\varphi^G, \psi^G] = \sum_{g \in R} [\varphi, (\psi^g_{H^g \cap H})^H]
             = \sum_{g \in R} [\varphi_{H^g \cap H}, \psi^g_{H^g \cap H}],
\]
where $R$ is a set of representatives of $H$-$H$-double cosets in $G$.

This condition implies the following:
For every $\varphi \in \Irr(H)$, there exists $g \in G$
such that $[\varphi_{H \cap H^g}, 1_{H \cap H^g}] > 0$.

The second condition also implies that,
for every linear character $\lambda$ of $H$,
there exists $g \in G$ such that $H \cap H^g$ is contained
in the kernel of $\lambda$.
\end{rem}

The following consequence will be useful in checking examples.

\begin{cor}\label{cond_ii_satisfied}
Let $G$ be a finite group and $H \leq G$.
Condition~(b)~(ii) in Proposition~\ref{diam3_equiv} is satisfied
if one of the following holds.

(i) There is $g \in G$ such that $|H^g \cap H| = 1$.

(ii) $H$ is core-free in $G$,
     and all nontrivial elements of $H$ are conjugate in $N_G(H)$.
\end{cor}

\begin{proof}
(i) Use the reformulation from Remark~\ref{reformulate},
and take $g$ with $|H^g \cap H| = 1$ as one element in $R$,
then $[\varphi_{H^g \cap H}, \psi^g_{H^g \cap H}] \neq 0$.

(ii)
We may assume that $H$ is nontrivial.
By Brauer's Permutation Lemma,
all nontrivial elements in $\Irr(H)$ are conjugate in $N_G(H)$.
Since $\Gamma(G, H)$ is connected,
there exists a nontrivial character $\varphi \in \Irr(H)$ such that
$[1_H^G, \varphi^G] > 0$.
Then $[1_H^G, \psi^G] > 0$ for all $\psi \in \Irr(H)$.
Since also $[\varphi^G, \psi^G] = [\varphi^G, \varphi^G] > 0$
for all nontrivial $\varphi, \psi \in \Irr(H)$
the result follows.
\end{proof}

An example where condition~(b)~(ii) in Proposition~\ref{diam3_equiv}
is satisfied but Corollary~\ref{cond_ii_satisfied} cannot be applied
is the group $G = S_3 \times S_3$
where $H$ is a non-normal $S_3$ type subgroup.
Note that two different $G$-conjugates of $H$ intersect in a subgroup
of order $2$ or $3$.

%

Now we record some easy consequences of Proposition~\ref{diam3_equiv}.
Several results about diameter three subgroups hold in fact already if
condition~(b)~(i) of this proposition is satisfied.
We introduce the following notation.

\begin{definition}
A proper subgroup $H$ of a group $G$ is called \emph{rich in $G$} if
$[\chi_H, 1_H] \not= 0$ holds for all $\chi\in\Irr(G)$.
\end{definition}

This terminology is motivated by part~(i) of the following corollary.

\begin{cor}\label{cor_easy}
Let $H$ be a rich subgroup in a finite group $G$.
Then the following assertions hold:

\begin{itemize}
\item[(i)]
  Each character in $\Irr(G)$ is a constituent of
  the permutation character $1_H^G$.
  In particular, we have $[G:H] \geq \sum_{\chi \in \Irr(G)} \chi(1)$.
\item[(ii)]
  $H$ is core-free in $G$.
\item[(iii)]
  The derived subgroup $G'$ of $G$ contains $H$;
  moreover, if $G' = H$ then $G$ is abelian.
\item[(iv)]
  If $\chi$ is an irreducible character of $G$ of degree $2$ then
  the kernel of $\chi$ contains $H'$.
\end{itemize}
\end{cor}

\begin{proof}
(i) follows immediately from the definition of richness.

(ii) follows from the fact that $\Gamma(G, H)$ is connected,
     and Proposition~\ref{corefree}.

(iii) follows from the definition of richness since $[\lambda_H, 1_H] \neq 0$
     for every linear character $\lambda$ of $G$,
     and that $G' = H$ happens only if $H$ is trivial, by part~(ii).

(iv) Suppose that $\chi \in \Irr(G)$ satisfies $[\chi_H, 1_H] \not= 0$
 and that $\chi(1) = 2$.
 Then $\chi_H$ must be the sum of two linear characters of $H$.
 Thus $H'$ is contained in the kernel of $\chi$.
\end{proof}

Part~(i) of Corollary~\ref{cor_easy} implies that
rich subgroups must be ``small'',
see Section~\ref{sect:large}.

\begin{prop}\label{reduction}
Let $G$ be a finite group with subgroups $1 < L \leq H < K \leq G$.
If $H$ is rich in $K$ then $L$ is rich in $G$,
and if $\Gamma(K, H)$ has diameter $3$
then $\Gamma(G, L)$ has diameter $3$.
\end{prop}

\begin{proof}
Let $H$ be rich in $K$, let $\chi \in \Irr(G)$,
and let $\eta$ be a constituent of $\chi_K$.
Then $[\chi_L, 1_L] \geq [\chi_H, 1_H] \geq [\eta_H, 1_H] > 0$,
and $L$ is rich in $G$.
Suppose that $\Gamma(K, H)$ has diameter $3$,
let $\alpha, \beta \in \Irr(L)$,
and let $\varphi, \psi \in \Irr(H)$ such that
$[\alpha^H, \varphi] \neq 0 \neq [\beta^H, \psi]$.
Then
$[\alpha^G, \beta^G] \geq [\varphi^G, \psi^G] \geq [\varphi^K, \psi^K] > 0$,
and $\Gamma(G, L)$ has diameter $3$.
\end{proof}

By Proposition~\ref{reduction},
a finite group $G$ contains a diameter three subgroup
if and only if $G$ contains a diameter three subgroup of prime order.
By Corollary~\ref{cond_ii_satisfied}~(i),
this is the case if and only if $G$ contains a rich subgroup of prime order,
which is the case if and only if $G$ contains a nontrivial rich subgroup.
Note also that this property can be decided from the character table of $G$.


\begin{lem}\label{lem:reduction}
Let $H$ be a rich subgroup in a finite group $G$,
and let $U$ be a subgroup of $G$ such that $G = HU$.
Then $U \cap H$ is rich in $U$.
\end{lem}

\begin{proof}
Let $\chi\in\Irr(U)$.
By Frobenius reciprocity and Mackey decomposition, we have
\[
   [\chi_{U \cap H}, 1_{U \cap H}] = [(\chi_{U \cap H})^H, 1_H]
                                   = [(\chi^G)_H, 1_H],
\]
and the right hand side is nonzero because any constituent $\psi$ of $\chi^G$
satisfies $[\psi_H, 1_H] > 0$, by the assumption that $H$ is rich in $G$.
\end{proof}

The existence of a diameter three subgroup in the group $G$ does in general not
imply the existence of a diameter three subgroup in the factor group of $G$
modulo a normal subgroup,
see for example Corollary~\ref{cor:supersolvable} below.
However, the following holds.

\begin{lem}\label{lem:factor}
\begin{itemize}
\item[1.]
  Let $H$ be a rich subgroup in $G$,
  and let $N$ be a proper normal subgroup of $G$.
  Then $HN/N$ is a rich subgroup in $G/N$.
\item[2.]
  Let $H$ be a diameter three subgroup of $G$, with $H$ of prime order.
  If $N$ is a normal subgroup of $G$ that does not contain $H$
  then $HN/N$ is a diameter three subgroup of $G/N$.
\end{itemize}
\end{lem}

\begin{proof}
In order to prove part 1, let $\chi \in \Irr(G/N)$,
and view $\chi$ as a character of $G$.
Since $[\chi_H,1_H] > 0$ the subgroup $H$ fixes a vector $v \neq 0$
in a module $V$ affording $\chi$.
Since $N$ acts trivially on $V$, the vector $v$ is also fixed by $HN$,
so that $[\chi_{HN}, 1_{HN}] > 0$.
Also, choosing $\chi$ nontrivial we see that we cannot have $HN = G$.
This shows that $HN/N$ is rich in $G/N$.

Part 2.~is just a special case of part 1.~where $H$ has prime order,
by Corollary~\ref{cond_ii_satisfied}.
\end{proof}

\begin{cor}\label{cor:supersolvable}
Let $H$ be a nontrivial subgroup of a supersolvable group $G$.
Then $H$ is not rich in $G$,
and $\Gamma(G, H)$ does not have diameter $3$.
\end{cor}

\begin{proof}
Let $G$ be a counterexample of minimal order,
and let $N$ be a minimal normal subgroup of $G$.
Then $N$ has prime order, so that $N < G$.
Since $H$ is core-free in $G$, $N$ does not contain $H$.
Thus $HN/N$ is a nontrivial rich subgroup of the supersolvable group $G/N$,
by Lemma~\ref{lem:factor}.
This contradicts the minimality of $G$.
\end{proof}


\begin{cor}\label{cor:maximal-in-simple}
Let $H$ be a rich subgroup in $G$.
If $H$ is maximal in $G$ then $G$ is simple.
\end{cor}

\begin{proof}
If $H$ is maximal in $G$
and $N$ is a nontrivial proper normal subgroup of $G$
then either $N \leq H$ or $HN = G$ holds.
The former cannot happen because rich subgroups are core-free.
The latter cannot happen because of Lemma~\ref{lem:factor}.
\end{proof}

See Section~\ref{sect:large} for examples of rich subgroups
that are maximal in simple groups.

In certain situations, one can go down from a finite group $G$
with a rich subgroup to a smaller group with a rich subgroup.

\begin{lem}\label{lem:godown}
Let $G$ be a finite group with subgroups $H < K < G$.

(i) Suppose that, for $\eta \in \Irr(K)$,
there are $m_\eta \in \mathbb{N}$ and $\chi_\eta \in \Irr(G)$
such that $(\chi_\eta)_K = m_\eta \eta$.
If the pair $(G, H)$ satisfies
Condition~(b)(i) in Proposition~\ref{diam3_equiv}
then the pair $(K, H)$ also satisfies this condition.

(ii) Suppose that, for $\chi \in \Irr(G)$,
there are $n_\chi \in \mathbb{N}$ and $\eta_\chi \in \Irr(K)$
such that $\chi_K = n_\chi \eta_\chi$.
If the pair $(G, H)$ satisfies
Condition~(b)(ii) in Proposition~\ref{diam3_equiv}
then the pair $(K, H)$ also satisfies this condition.
\end{lem}

\begin{proof}
(i) Let $\eta \in \Irr(K)$, and let $m_\eta$ and $\chi_\eta$ be as above.
Then Condition~(b)(i) for $(G, H)$ implies:
$0 < [(\chi_\eta)_H, 1_H] = [m_\eta \eta_H, 1_H]$,
so that $[\eta_H, 1_H] > 0$.
Thus Condition~(b)(i) is satisfied for $(K, H)$.

(ii) Let $\varphi, \psi \in \Irr(H)$.
Then, by Condition~(b)(ii) for $(G, H)$,
$\varphi^G$ and $\psi^G$ have a common constituent $\chi$.
Let $n_\chi$ and $\eta_\chi$ be as above.
Then we have
$0 < [\chi, \varphi^G] = [\chi_H, \varphi] = [n_\chi (\eta_\chi)_H, \varphi]
   = n_\chi [\eta_\chi, \varphi^K]$,
so that $[\eta_\chi, \varphi^K] > 0$ and, similarly,
$[\eta_\chi, \psi^K] > 0$.
We conclude that $[\varphi^K, \psi^K] > 0$,
and Condition~(b)(ii) is satisfied for $(K, H)$.
\end{proof}

\begin{rem}
(a) We note that the hypothesis of~(i) is satisfied
in the special case where every irreducible character of $K$
extends to an (irreducible) character of $G$.
In this case, $K$ is sometimes called a CR-subgroup of $G$
where CR stands for ``character restriction'';
see for example~\cite{Isaacs}.
In particular,
(i) applies whenever $K$ has a normal complement in $G$.

(b) Similarly, the hypothesis in~(ii) is satisfied in the special case
where every irreducible character of $G$ restricts to
an irreducible character of $K$.
\end{rem}

\begin{lem}\label{lem:extendcentre}
Let $Z$ be a central subgroup of a finite group $K$,
and suppose that we have subgroups $1 < H < G \leq K = GZ$.
Then $H$ is rich in $G$ if and only if $H$ is rich in $K$,
and the diameter of $\Gamma(G, H)$ is $3$ if and only if
the diameter of $\Gamma(K, H)$ is $3$.
\end{lem}

\begin{proof}
The direction from $(G, H)$ to $(K, H)$ follows from
Proposition~\ref{reduction}.

For the other direction,
it is easy to see that the irreducible characters of $K$ are extensions
of the irreducible characters of $G$.
(Note that $K$ is isomorphic to a quotient of $G \times Z$.)
Now apply Lemma~\ref{lem:godown}.
\end{proof}

\begin{prop}\label{prop:isoclinic}
Let $H_1$ be a proper subgroup of a finite group $G_1$,
and suppose that $H_1$ is rich in $G_1$
(or $\Gamma(G_1, H_1)$ has diameter $3$).
Moreover, let $G_2$ be a finite group which is isoclinic to $G_1$.
Then there exists a proper subgroup $H_2$ of $G_2$
such that $H_2$ is rich in $G_2$
(or $\Gamma(G_2, H_2)$ has diameter $3$).
\end{prop}

\begin{proof}
Since $G_1$ and $G_2$ are isoclinic,
there exists a finite group $K$ containing subgroups isomorphic to $G_1$
and $G_2$ (which we identify with $G_1$ and $G_2$)
such that $G_1' = G_2'$ and $K = G_i Z_i$ with a central subgroup $Z_i$
of $K$, for $i = 1, 2$.
(This characterization of isoclinism is often attributed to Conway,
see~\cite[Section~6.7]{CCN85}.
A proof that it is equivalent to the usual definition can be found
in~\cite[Theorem~4.2]{Hekster}.)

We note that $H_1$ is contained in $G_1' = G_2'$
by Corollary~\ref{cor_easy}~(iii).
Thus we can view $H_2 := H_1$ as a subgroup of $G_2$ as well.
Now we apply Lemma~\ref{lem:extendcentre} twice:
If $\Gamma(G_1, H_1)$ has diameter $3$
then $\Gamma(K, H_1)$ has diameter $3$,
and therefore $\Gamma(G_2, H_1)$ has diameter $3$.
\end{proof}

\section{Examples}\label{sect:examples}

\begin{example}\label{expl_frob}
Let $G:= \AGL(1, p^n)$ be the affine group of degree $1$ over a field
with $p^n$ elements where $p$ is a prime and $n \geq 2$ is an integer.
Then $G$ is a Frobenius group with an elementary abelian kernel $E$
of order $p^n$ and a cyclic complement $C$ of order $p^n-1$.
Moreover, $\Irr(G)$ consists of $p^n-1$ linear characters
and one faithful character $\chi$ of degree $p^n-1$
(cf.~\cite[Satz~V.16.13]{HuppertI}).
Let $H$ be a subgroup of order $p$ in $G$.
The linear characters of $G$ are trivial on $E$ and thus on $H$,
and $\chi_E$ is the sum of all nontrivial irreducible characters of $E$.
Since $n \geq 2$,
every irreducible character of $H$ is a constituent of $\chi_H$.
Thus $H$ is rich in $G$.
Since also condition~(b)~(ii) is satisfied by
Corollary~\ref{cond_ii_satisfied}~(i),
$\Gamma(G, H)$ has diameter $3$.
\end{example}

\begin{example}\label{expl_frob_sub}
Let $G$ be the subgroup of order $p^2 d$ in $\AGL(1, p^2)$,
where $d > 1$ and $d$ divides $p^2-1$.
As in Example~\ref{expl_frob}, $G$ is a Frobenius group
with an elementary abelian kernel $E$
of order $p^2$ and a cyclic complement $C$ of order $d$.
Then $G$ has a diameter three subgroup (necessarily of order $p$)
if and only if $d$ is divisible by $(p+1) (p-1)_2$,
where $(p-1)_2$ is the $2$-part of $p-1$.
Equivalently, this happens if $(p^2-1)/d$ is an odd divisor of $p-1$.

First note that $d$ is divisible by $(p+1) (p-1)_2$ if and only if
all $p+1$ subgroups of order $p$ in $E$ are conjugate in $G$.
Note that $|N_C(H)| = \gcd(d, p-1)$, where $H$ is a subgroup of order $p$.
In order to see this, observe that on the one hand,
if $d$ is divisible by $(p+1) (p-1)_2$ then $\gcd(d, p-1) = d/(p+1)$,
so $|C/N_C(H)| = p+1$ holds,
and on the other hand,
if there is only one class of order $p$ subgroups in $G$
then $|N_C(H)| = d/(p+1)$, which means that $(p-1)_2$ divides $d/(p+1)$.

Let $H$ be a subgroup of order $p$ in $E$.
Then $G$ has $(p^2-1)/d$ nonlinear irreducible characters of degree $d$.
The restriction of each such character to $E$ is a sum of $d$
nontrivial irreducibles of $E$.

If there is only one class of order $p$ subgroups in $G$ then
the transitive action of $C$ on these $p+1$ subgroups means that
each $C$-orbit of nontrivial irreducibles of $E$ contains at least one
character with kernel $H$, thus $H$ is a diameter three subgroup of $G$.

Conversely, let $H_1, H_2, \ldots, H_k$ be representatives of conjugacy
classes of subgroups of order $p$ in $G$.
Each $C$-orbit of nontrivial irreducibles of $E$ contains characters
with kernels only from one class of subgroups of order $p$.
Considering an orbit containing no character with kernel $H_i$ yields
that $H_i$ cannot be a diameter three subgroup.
\end{example}

\begin{example}\label{expl_frob_sub_n}
Let $G$ be the subgroup of order $p^n d$ in $\AGL(1, p^n)$,
where $d$ divides $p^n-1$.
Then $G$ acts on the set of $(p^n-1)/(p-1)$ subgroups of order $p^{n-1}$
in $G$ by conjugation.
They are the kernels of the nontrivial characters of the normal subgroup $E$
of order $p^n$ in $G$.
Let $U$ be a subgroup of order $p^{n-1}$ in $G$.
Its normalizer has order $p^n \gcd(d, p-1)$.
Thus $U$ has precisely $d / \gcd(d, p-1)$ conjugates in $G$.
Hence there are precisely $(p^n-1) \gcd(d, p-1)/((p-1) d)$
conjugacy classes of subgroups of order $p^{n-1}$ in $G$.
If we can choose one subgroup from each of these conjugacy classes
such that their intersection is nontrivial
then this intersection contains a diameter three subgroup of order $p$ in $G$.
Since the intersection of $k$ subgroups of order $p^{n-1}$ yields
a subgroup of order at least $p^{n-k}$,
such a choice is possible if $n > (p^n-1) \gcd(d, p-1)/((p-1) d)$.

This implies that if $n \geq 3$ and $p$ is odd then
the subgroup of index $2$ in $\AGL(1, p^n)$
has a diameter three subgroup of order $p$.
Example~\ref{expl_frob_sub} shows that this is not the case for $n = 2$.

(The bound is not sharp.
For example, the subgroup of order $7^3 \cdot 19$ in $\AGL(1,7^3)$
has $3$ classes of subgroups of order $7^2$,
and still has a diameter three subgroup.)
\end{example}

\begin{example}\label{simple_group}
Let $H$ be a subgroup of order $2$ in a nonabelian finite simple group $G$.
Then $\Gamma(G, H)$ has diameter $3$.
In order to see this, we check that for any $\chi \in \Irr(G)$,
$[\chi_H, 1_H] \neq 0$ holds.
Write $H = \{1, h\}$.
If $\chi \in \Irr(G)$ satisfies
$[\chi_H, 1_H] = 0$ then $\chi(h) = -\chi(1)$,
and $h$ is contained in $Z(\chi)$, the center of $\chi$.
In particular, we have $Z(\chi) \neq 1$.
Since $G$ is simple, this implies $Z(\chi) = G$.
Since $Z(\chi)/ \ker(\chi)$ is cyclic we conclude that $\ker(\chi) = G$.
Thus $\chi = 1_G$ which, however, is impossible.
\end{example}

\begin{rem}
Example~\ref{simple_group} shows that nonabelian finite simple groups
always have diameter three subgroups.
This fact does not generalize to quasisimple finite groups;
in fact, the quasisimple group $\SL(2, 5) = 2.A_5$ does not have a
diameter three subgroup.
This can be seen by noting that $\SL(2, 5)$ is a Frobenius complement;
thus it cannot have nontrivial rich subgroups,
by~\cite[(25.5)]{Feit}.

See Section~\ref{sect:quasisimple} for more about quasisimple groups.
\end{rem}

\begin{prop}\label{simplegroup}
Every nonabelian finite simple group has a solvable subgroup
which contains a diameter three subgroup of order $2$.
\end{prop}

\begin{proof}
It is known (cf.~the main result of~\cite{BW})
that every nonabelian finite simple group contains a minimal simple group,
i.~e., a nonabelian simple group all of whose proper subgroups are solvable.
Thus it suffices to prove that every minimal simple group
has a proper subgroup which contains a diameter three subgroup of order $2$.
The minimal simple groups were classified by Thompson;
they are given as follows, cf.~\cite[Bemerkung II.7.5]{HuppertI}:

\[
   \begin{array}{ll}
    \PSL(2, p),
       & \textrm{$p > 3$ a prime with $p^2-1 \not\equiv 0 \pmod{5}$}, \\
    \PSL(2, 2^q), & \textrm{$q$ a prime}, \\
    \PSL(2, 3^q), & \textrm{$q$ an odd prime}, \\
    \PSL(3, 3), & \\
    \Sz(2^q), & \textrm{$q$ an odd prime}.
   \end{array}
\]

By part~(4) of \cite[Satz II.8.27]{HuppertI},
the groups $\PSL(2, p)$, where $p > 3$ is a prime,
and the groups $\PSL(2, 3^q)$ contain subgroups isomorphic to the alternating
group $A_4$ which has a diameter three subgroup of order $2$.
Part~(7) of the same theorem yields Frobenius groups with Frobenius kernel
of $2$-power order, as in Example~\ref{expl_frob},
as subgroups of $\PSL(2, 2^q)$.
The question about $\PSL(3, 3)$ can be answered computationally;
this group has maximal subgroups of the type $S_4$,
and hence subgroups isomorphic to $A_4$.
Finally, the Suzuki group $G = \Sz(2^q)$ has a Sylow $2$-subgroup $P$
of order $2^{2q}$, with an elementary abelian center $Z$ of order $2^q$
(cf.~\cite[Lemma XI.3.1]{HBIII}).
The normalizer of $P$ in $G$ is a semidirect product of $P$ with a cyclic
group of order $2^q-1$,
it contains a subgroup $S$ which is the semidirect product of $Z$ with
the cyclic group of order $2^q-1$.
This group $S$ is a Frobenius group with Frobenius kernel of $2$-power order,
as in Example~\ref{expl_frob}.
\end{proof}

\section{Large rich subgroups}\label{sect:large}

Part~(i) of Corollary~\ref{cor_easy} implies that
rich subgroups must be ``small'',
Here is a quantitative version of this statement.

\begin{prop}\label{sqrtbound}
Let $T(G) = \sum_{\chi} \chi(1)$, where $\chi$ runs over $\Irr(G)$,
$k(G) = |\Irr(G)|$,
and $b(G) = \max\{\chi(1); \chi\in\Irr(G) \}$.

(i) We have
\begin{eqnarray*}
   |G| & = & \sum_{\chi\in\Irr(G)} \chi(1)^2 \\
       & = & \sum_{\chi\in\Irr(G)} \chi(1) \cdot b(G) - 
          \sum_{\chi\in\Irr(G)} \chi(1) \cdot \left( b(G)-\chi(1) \right) \\
       & = & T(G) \cdot b(G) - [G:G'] \cdot \left( b(G)-1 \right) -
          \sum_{1 < \chi(1) < b(G)} \chi(1) \cdot \left( b(G)-\chi(1) \right).
\end{eqnarray*}
(ii) We have
\[
   |G| \leq T(G) \cdot b(G) \leq T(G) \cdot \left(T(G) - k(G) + 1\right).
\]
Moreover, we have $|G| = T(G) b(G)$ if and only if $G$ is abelian.

(iii) If $1 \not= H \leq G$ is rich in $G$ then $|H| \leq [G:H] - k(G) + 1$
and $|H| < b(G)$.
In particular, $|H| < \sqrt{|G|}$.
\end{prop}

\begin{proof}
Part~(i) is clear, part~(ii) follows easily from part~(i),
and part (iii) follows from~(ii) since $T(G) \leq [G:H]$ by
Corollary~\ref{cor_easy},
%
\end{proof}

\begin{prop}\label{no-prime-power-index}
Let $H$ be a rich subgroup in a finite group $G$,
and suppose that $[G:H]$ is a power of a prime $p$.
Then $H = 1$.
\end{prop}

\begin{proof}
Let $G$ be a minimal counterexample.
Then $G$ is not a $p$-group, by Corollary~\ref{cor:supersolvable}.
Let $P$ be a Sylow $p$-subgroup of $G$.
Then $G = HP$ since $[G:H]$ and $[G:P]$ are coprime.
By Lemma~\ref{lem:reduction}, $H \cap P$ is rich in $P$.
Thus Corollary~\ref{cor:supersolvable} implies that $H \cap P = 1$.
Hence $|G| = |H| \cdot |P|$,
so that $H$ is a Hall $p'$-subgroup of $G$.
Thus $1_H^G$ is the character of the projective cover of
the trivial module (in characteristic $p$).
Thus its constituents lie in the principal $p$-block of $G$.
We conclude that $G$ has only one $p$-block.
Let $N$ be a minimal normal subgroup of $G$.
Then $N$ is isomorphic to $S^k$ where $k$ is a positive integer
and $S$ is a simple group.
Moreover, $N$ has a unique $p$-block, and $S$ has a unique $p$-block;
in particular, $p$ divides $|S|$.
If $S$ is nonabelian then, as is well-known, we have $p = 2$,
and $S$ is isomorphic to $M_{22}$ or $M_{24}$.
On the other hand, $N$ and $S$ both have Hall $p'$-subgroups
which is a contradiction.
This shows that $S$ is abelian, i.~e.,
$N$ is an elementary abelian $p$-group.
By Lemma~\ref{lem:factor}, $HN/N$ is rich in $G/N$ and of $p$-power index.
Since $|G/N| < |G|$ this implies $HN/N = 1$, i.~e., $1 < H < N$.
Then $G/N$ and $G$ are $p$-groups, and we have a contradiction.
\end{proof}

\begin{rem}
If $1 \not= H \leq G$ is rich in $G$ then
$1_H^G = \sum_\chi a_\chi \chi$, with $a_\chi > 0$ for all $\chi\in\Irr(G)$,
hence the number $[1_H^G, 1_H^G] = \sum_\chi a_\chi^2$
of $H$-$H$-double cosets in $G$,
which is equal to the rank of the permutation action of $G$
on the cosets of $H$,
is at least equal to $k(G)$.
Note that Proposition~\ref{sqrtbound} only yields that this rank is
at least $3$, because rank $2$ would imply a doubly transitive action of $G$
and hence $|G| \geq [G:H]([G:H]-1)$.
\end{rem}

\begin{example}
Let $G:= \AGL(1, 2^n)$ be the group from Example~\ref{expl_frob}
in the special case $p = 2$,
but now choose a subgroup $H$ of order $p^{n-1}$ in $G$.
Then the linear irreducibles of $G$ restrict to $1_H$,
and because $T(G) = 2(2^n - 1) = [G:H]$ holds,
also the unique nonlinear character $\chi$ of $G$ occurs with multiplicity
$1$ in $1_H^G$, that is, $1_H^G$ is exactly the sum of $\Irr(G)$.
In order to verify condition~(b)~(ii) and hence to show
that $H$ is indeed a diameter three subgroup of $G$,
we note that $\chi$ is a constituent of each $\varphi^G$,
for $\varphi\in\Irr(H)$.

Now we may take the direct product of $G$ with an abelian group
(or more generally, take a group isoclinic with $G$),
and keep the subgroup $H$,
then we get again that $H$ is a diameter three subgroup of $G$
with the property that $[G:H] = T(G)$ holds,
that is, $H$ has the largest possible order.
Moreover, the rank of the permutation action of $G$ on the cosets of $H$
is exactly $k(G)$.
\end{example}

%
%

\begin{rem}\label{rem:smallpermdegree}
If $H < G$ is rich in $G$
then the action of $G$ on the cosets of $H$ is faithful.
Hence there are, for each prescribed integer $n$, only finitely many pairs
$(G, H)$ such that $H$ is a rich subgroup of index at most $n$ in $G$.
Table~\ref{table:smallpermdegree} lists all groups $G$ with
a nontrivial rich subgroup of index at most $45$.
It was computed using the list of all groups of order at most $2\,000$,
up to isomorphism, that contain a nontrivial rich subgroup,
see Section~\ref{subsect:minsmall}.
The columns show $n = [G:H]$, $|G|$, the number $i$ such that
$G$ can be obtained as the $i$-th group of its order,
according to the numbering in~\cite{SmallGroup},
a structure description of $G$,
and a $+$ sign if the point stabilizer $H$ is a diameter three subgroup
of $G$
--the Frobenius group of order $3^3 \cdot 13 = 351$ is the only example
where this is not the case, see Remark~\ref{independent}.

\begin{table}[htpb]
\begin{center}
\small
\begin{tabular}{|r|r|r|l|c|} \hline
 $n$ & $|G|$ &   $i$ &                    $G$ & diam. 3? \\
\hline
 $6$ &  $12$ &   $3$ &                  $A_4$ & + \\
$12$ &  $24$ &  $12$ &                  $S_4$ & + \\
     &  $24$ &  $13$ &         $2 \times A_4$ & + \\
$14$ &  $56$ &  $11$ &            $2^3\spl 7$ & + \\
$18$ &  $36$ &   $3$ &            $2^2\spl 9$ & + \\
     &  $36$ &  $11$ &         $3 \times A_4$ & + \\
$20$ &  $60$ &   $5$ &                  $A_5$ & + \\
     &  $80$ &  $49$ &            $2^4\spl 5$ & + \\
$24$ &  $48$ &   $3$ &            $4^2\spl 3$ & + \\
     &  $48$ &  $30$ &            $A_4\spl 4$ & + \\
     &  $48$ &  $31$ &         $4 \times A_4$ & + \\
     &  $48$ &  $48$ &         $2 \times S_4$ & + \\
     &  $48$ &  $49$ &       $2^2 \times A_4$ & + \\
     &  $48$ &  $50$ &            $2^4\spl 3$ & + \\
     &  $72$ &  $39$ &            $3^2\spl 8$ & + \\
     &  $72$ &  $41$ &          $3^2\spl Q_8$ & + \\
     &  $96$ &  $70$ &    $(2^4\spl 2)\spl 3$ & + \\
     &  $96$ &  $71$ &    $(4^2\spl 2)\spl 3$ & + \\
$28$ &  $56$ &  $11$ &            $2^3\spl 7$ & + \\
     & $112$ &  $41$ &   $2 \times 2^3\spl 7$ & + \\
$30$ &  $60$ &   $5$ &                  $A_5$ & + \\
\hline
\end{tabular}
\ \ \
\begin{tabular}{|r|r|r|l|c|} \hline
 $n$ & $|G|$ &   $i$ &                    $G$ & diam. 3? \\
\hline
     &  $60$ &   $9$ &         $5 \times A_4$ & + \\
     & $240$ & $191$ &           $2^4\spl 15$ & + \\
$36$ &  $72$ &  $15$ &    $(2^2\spl 9)\spl 2$ & + \\
     &  $72$ &  $16$ &   $2 \times 2^2\spl 9$ & + \\
     &  $72$ &  $42$ &         $3 \times S_4$ & + \\
     &  $72$ &  $43$ & $(3 \times A_4)\spl 2$ & + \\
     &  $72$ &  $44$ &       $A_4 \times S_3$ & + \\
     &  $72$ &  $47$ &         $6 \times A_4$ & + \\
     & $144$ & $184$ &       $A_4 \times A_4$ & + \\
$39$ & $351$ &  $12$ &           $3^3\spl 13$ & - \\
$40$ &  $80$ &  $49$ &            $2^4\spl 5$ & + \\
     & $120$ &  $34$ &                  $S_5$ & + \\
     & $120$ &  $35$ &         $2 \times A_5$ & + \\
     & $160$ & $234$ &    $(2^4\spl 5)\spl 2$ & + \\
     & $160$ & $235$ &   $2 \times 2^4\spl 5$ & + \\
$42$ &  $84$ &  $10$ &         $7 \times A_4$ & + \\
     &  $84$ &  $11$ &  $(14 \times 2)\spl 3$ & + \\
     & $168$ &  $42$ &            $\PSL(3,2)$ & + \\
     & $168$ &  $43$ &    $2^3\spl (7\spl 3)$ & + \\
     & $168$ &  $44$ &   $3 \times 2^3\spl 7$ & + \\
     &       &       &                        &   \\
\hline
\end{tabular}
\caption{Groups having nontrivial rich subgroups of index at most $45$}
\label{table:smallpermdegree}
\end{center}
\end{table}
\end{rem}

\begin{prop}\label{prop:index2p}
Let $H$ be a nontrivial rich subgroup of index $2 p$ in a finite group $G$
where $p$ is a prime.
Then $p$ is a Mersenne prime,
and $G$ is a Frobenius group of order $p(p+1)$.
\end{prop}

\begin{proof}
By Proposition~\ref{no-prime-power-index}, $p$ is odd.
Let $P$ be a Sylow $p$-subgroup of $G$.
If $P$ is normal in $G$ then Lemma~\ref{lem:factor} implies that
either $H \leq P$ or
$HP/P$ is a rich subgroup of index $2$ in $G/P$.
The former case cannot occur because then $H$ is a $p$-group,
and $|G| < 4 p^2$ implies $|H| = p$ and thus $|G| = 2 p^2$;
however, this means that $|H| < b(G) \leq 2$, hence $H$ is trivial.
The latter case cannot occur because then $HP/P$ is normal in $G/P$
which is a contradiction
since rich subgroups are core-free.
Thus $P$ is not normal in $G$.
By the It\^o-Michler theorem,
there is $\chi \in \Irr(G)$ such that $p$ divides $\chi(1)$.
Since $\chi(1) \leq b(G) < T(G) \leq [G:H] = 2 p$ we conclude
that $b(G) = \chi(1) = p$. Thus $|H| < b(G) < p+1 \leq
[G:N_G(P)] \leq [G:P] = 2|H| < 2p$.
Hence we obtain $P = N_G(P)$ and $|G| = p(p+1)$.
Since $N_G(P) = C_G(P)$,
a theorem of Burnside implies that $G$ has a normal $p$-complement $N$.
Since $|N| = p+1$, $P$ acts transitively on $N \setminus \{1\}$.
Since $p+1$ is even, $N$ has to be an elementary abelian $2$-group,
and $G$ is a Frobenius group with kernel $N$.
\end{proof}

\begin{rem}
The statement of Proposition~\ref{prop:index2p} generalizes to
the situation of nontrivial rich subgroups of index $p q$
where $p$ and $q$ are odd primes.
We hope to publish this result in a sequel.
\end{rem}

\begin{rem}
For a given group $G$, we can ask which of its subgroups $H$ are maximal
with the property that $H$ is rich in $G$.
Table~\ref{table:TOMresult} lists these subgroups $H$
(where duplicate isomorphism types have been removed)
for some small simple groups $G$.
The columns labelled by $n$, $g$, $m$ list
the total number of classes of subgroups of the group $G$,
the number of classes of rich subgroups $H$,
and the number of classes of rich subgroups $H$
of maximal order.
In fact, all subgroups listed in the table are diameter three subgroups.

\begin{table}[htpb]
\begin{center}
\small
\begin{tabular}{|l|r|r|r|p{8cm}|} \hline
$G$   & $n$ & $g$ & $m$ & $H$ \\ \hline
$A_5$ & $9$ & $2$ & $2$ &
$2$, $3$ \\
$\PSL(3,2)$ & $15$ & $3$ & $2$ &
$3$, $4$ \\
$A_6$ & $22$ & $9$ & $6$ &
$5$, $4$, $2^2$, $S_3$ \\
$\PSL(2,8)$ & $12$ & $4$ & $3$ &
$3$, $2^2$, $7$ \\
$\PSL(2,11)$ & $16$ & $7$ & $5$ &
$2^2$, $5$, $S_3$, $6$ \\
$\PSL(2,13)$ & $16$ & $7$ & $5$ &
$2^2$, $6$, $S_3$, $7$ \\
$\PSL(2,17)$ & $22$ & $10$ & $5$ &
$S_3$, $8$, $D_8$, $9$ \\
$A_7$ & $40$ & $16$ & $9$ &
$5$, $S_3$, $3^2$, $A_4$, $6 \times 2$, $3\spl 4$ \\
$\PSL(2,19)$ & $19$ & $10$ & $6$ &
$S_3$, $9$, $D_{10}$, $10$, $A_4$ \\
$\PSL(2,16)$ & $21$ & $11$ & $5$ &
$S_3$, $2^3$, $D_{10}$, $A_4$, $15$ \\
$\PSL(3,3)$ & $51$ & $14$ & $6$ &
$S_3$, $Q_8$, $8$, $A_4$, $3 \times S_3$ \\
$\PSU(3,3^2)$ & $36$ & $3$ & $2$ &
$3$, $4$ \\
$\PSL(2,23)$ & $23$ & $11$ & $5$ &
$S_3$, $D_8$, $11$, $12$ \\
$\PSL(2,25)$ & $37$ & $19$ & $9$ &
$D_8$, $D_{10}$, $D_{12}$, $A_4$, $12$, $13$ \\
$M_{11}$ & $39$ & $10$ & $7$ &
$2^2$, $5$, $S_3$, $6$, $Q_8$, $8$ \\
$\PSL(2,27)$ & $16$ & $10$ & $6$ &
$3^2$, $A_4$, $13$, $14$, $D_{14}$ \\
$\PSL(2,29)$ & $22$ & $12$ & $7$ &
$S_3$, $D_{10}$, $A_4$, $14$, $D_{14}$, $15$ \\
$\PSL(2,31)$ & $29$ & $15$ & $8$ &
$S_3$, $D_8$, $D_{10}$, $A_4$, $15$, $16$ \\
$\PSL(3,4)$ & $95$ & $60$ & $15$ &
$D_{10}$, $A_4$, $7\spl 3$, $S_4$, $4^2\spl 2$, $3^2\spl 4$ \\
$A_8$ & $137$ & $46$ & $17$ &
$7$, $D_8$, $D_{10}$, $A_4$, $6 \times 2$, $3\spl 4$, $D_{12}$, $2 \times D_8$, $(4 \times 2)\spl 2$, $3^2\spl 2$, $3 \times S_3$, $S_4$, $2 \times A_4$ \\
$\PSL(2,37)$ & $23$ & $13$ & $6$ &
$D_{12}$, $A_4$, $18$, $D_{18}$, $19$ \\
$\PSU(4,2^2)$ & $116$ & $20$ & $7$ &
$2^2$, $4$, $5$, $S_3$, $A_4$, $(4 \times 2)\spl 2$, $3 \times S_3$ \\
$Sz(8)$ & $22$ & $9$ & $6$ &
$5$, $7$, $4 \times 2$, $13$ \\
$\PSL(2,32)$ & $24$ & $16$ & $4$ &
$S_3$, $11$, $2^4$, $31$ \\
$\PSL(2,41)$ & $33$ & $21$ & $7$ &
$D_{14}$, $D_{20}$, $20$, $21$, $S_4$ \\
$\PSL(2,43)$ & $20$ & $12$ & $7$ &
$S_3$, $A_4$, $D_{14}$, $21$, $22$, $D_{22}$ \\
$\PSL(2,47)$ & $29$ & $19$ & $7$ &
$A_4$, $D_{12}$, $D_{16}$, $23$, $24$ \\
$\PSL(2,49)$ & $51$ & $33$ & $10$ &
$D_{10}$, $D_{16}$, $24$, $D_{24}$, $S_4$, $25$, $7\spl 6$ \\
$\PSU(3,4^2)$ & $34$ & $4$ & $3$ &
$3$, $4$, $5$ \\
$\PSL(2,53)$ & $20$ & $12$ & $6$ &
$A_4$, $D_{18}$, $26$, $D_{26}$, $27$ \\
$M_{12}$ & $147$ & $64$ & $13$ &
$11$, $A_4$, $D_{12}$, $2 \times D_8$, $3 \times S_3$, $2 \times A_4$, $8\spl 2^2$, $4^2\spl 2$, $(8\spl 2)\spl 2$, $2 \times 5\spl 4$ \\
$\PSL(2,59)$ & $26$ & $18$ & $7$ &
$D_{12}$, $A_4$, $D_{20}$, $29$, $30$, $D_{30}$ \\
$\PSL(2,61)$ & $32$ & $18$ & $7$ &
$A_4$, $D_{12}$, $D_{20}$, $30$, $D_{30}$, $31$ \\
$\PSU(3,5^2)$ & $80$ & $21$ & $9$ &
$5$, $S_3$, $8$, $3\spl 4$, $7\spl 3$, $\SL(2,3)$, $3 \times A_4$ \\
$\PSL(2,67)$ & $20$ & $12$ & $7$ &
$S_3$, $A_4$, $D_{22}$, $33$, $34$, $D_{34}$ \\
$J_1$ & $40$ & $30$ & $12$ &
$D_{12}$, $D_{20}$, $D_{22}$, $2 \times A_4$, $3 \times D_{10}$, $D_{30}$, $5 \times S_3$, $7\spl 6$, $11\spl 5$, $19\spl 3$, $A_5$ \\
\hline
\end{tabular}
\caption{Rich subgroups in simple groups, maximal w.~r.~t. inclusion}
\label{table:TOMresult}
\end{center}
\end{table}
\end{rem}

\begin{rem}
By Corollary~\ref{cor:maximal-in-simple},
a maximal subgroup $H$ of $G$ can be rich only if $G$ is simple.
Examples where this happens in simple groups $G$ of Lie type are listed in
Table~\ref{table:liemaxresult}.
Table~\ref{table:spormaxresult} lists all maximal subgroups $H$
in sporadic simple groups $G$ that are rich.
The subgroups shown in these tables are diameter three subgroups.

\begin{table}[htpb]
\begin{center}
{\small
\begin{tabular}{|l|l|} \hline
$G$            & $H$ \\ \hline
$\PSL(2,27)$   & $A_4$ \\
$\PSL(2,109)$  & $A_5$, $A_5$ \\
$\PSL(2,113)$  & $S_4$, $S_4$ \\
$\PSL(2,125)$  & $A_5$ \\
$\PSL(3,7)$    & $3^2\spl Q_8$, $19\spl 3$ \\
$\PSU(3,8^2)$  & $19\spl 3$ \\
$\PSL(3,8)$    & $7^2\spl S_3$, $\PSL(3,2)$ \\
$\PSU(3,11^2)$ & $37\spl 3$ \\
${}^2G_2(27)$  & $(2^2 \times D_{14})\spl 3$, $19\spl 6$ \\
\hline
\end{tabular}
}
\caption{Rich maximal subgroups in Lie type simple groups}
\label{table:liemaxresult}
\end{center}
\end{table}

\begin{table}[htpb]
\begin{center}
{\small
\begin{tabular}{|l|p{12cm}|} \hline
$G$      & $H$ \\ \hline
$J_1$    & $7\spl 6$ \\
$Suz$    & $A_7$ \\
$ON$     & $3^4\spl 2^{1+4}D_{10}$, $M_{11}$, $M_{11}$, $A_7$, $A_7$ \\
$Fi22$   & $M_{12}$ \\
$Ly$     & $67\spl 22$, $37\spl 18$ \\
$Th$     & $A_5.2$ \\
$J4$     & $\PGL(2,23)$, $\PSU(3, 3^2)$, $29\spl 29$, $43\spl 14$,
           $37\spl 12$ \\
$F_{3+}$ & $7\spl 6 \times A_7$, $\PGL(2,13)$, $\PGL(2,13)$, $29\spl 14$ \\
$B$      & $\PSL(2,49).2_3$, $\PSL(2,31)$, $\PSL(3,3)$, $\PGL(2,17)$,
           $\PGL(2,11)$, $47\spl 23$ \\
$M$      & $(7\spl 3 \times He)\spl 2$,
           $(5^2\spl [2^4] \times \PSU(3, 5^2)).S_3$,
           $7^{2+1+2}\spl \GL(2,7)$, $(S_5 \times S_5 \times S_5)\spl S_3$,
           $(\PSL(2,11) \times \PSL(2,11))\spl 4$,
           $(7^2\spl (3 \times 2A_4) \times \PSL(2,7)).2$,
           $(13\spl 6 \times \PSL(3,3)).2$, $\PSU(3,4^2).4$, $\PSL(2,71)$,
           $\PSL(2,59)$, $11^2\spl (5 \times 2.A_5)$, $\PSL(2,41)$,
           $\PGL(2,29)$,
           $7^2\spl \SL(2,7)$, $\PGL(2,19)$, $\PGL(2,13)$, $41\spl 40$ \\
\hline
\end{tabular}
}
\caption{Rich maximal subgroups in sporadic simple groups}
\label{table:spormaxresult}
\end{center}
\end{table}
\end{rem}

\section{Quasisimple groups}\label{sect:quasisimple}

The aim of this section is to classify those quasisimple groups,
i.~e., perfect central extensions of simple groups,
which have a diameter three subgroup.

\begin{thm}\label{thm_quasisimple}
Let $G$ be a quasisimple group.
Then $G$ has a diameter three subgroup,
except if $G \cong \SL(2, 5) \cong 2.A_5$ or $G \cong \SL(2, 9) \cong 2.A_6$
or $G \cong 6.A_6$.
\end{thm}

The proof uses the classification of finite simple groups.

The idea is to either establish directly the existence of
a diameter three subgroup,
or to prove the existence of a proper
quasisimple subgroup which is already known to have a
diameter three subgroup.

We start with computational results in cases where we do not know
a conceptual approach.

\begin{lem}\label{quasisimple_compute}
Let $G$ be a perfect central extension of a simple group that is
either sporadic simple or has an exceptional Schur multiplier.
Then $G$ has a diameter three subgroup,
except in the exceptional cases of Theorem~\ref{thm_quasisimple}.
\end{lem}

\begin{proof}
The character tables of all groups $G$ in question are available in the
Character Table Library~\cite{CTblLib},
and checking the conditions of Proposition~\ref{diam3_equiv} requires only
the character table of the group.

\begin{table}[htpb]
\begin{center}
\begin{tabular}{|lcl|l|} \hline
$S$          &         &                & $e$          \\ \hline
             &         & $A_6$          & $3$          \\
             &         & $A_7$          & $3$          \\
$A_1(4)$     & $\cong$ & $\PSL(2, 4)$   & $2$          \\
$A_1(9)$     & $\cong$ & $\PSL(2, 9)$   & $3$          \\
$A_2(2)$     & $\cong$ & $\PSL(3, 2)$   & $2$          \\
$A_2(4)$     & $\cong$ & $\PSL(3, 4)$   & $4 \times 4$ \\
$A_3(2)$     & $\cong$ & $\PSL(4, 2)$   & $2$          \\
${}^2A_3(2)$ & $\cong$ & $\PSU(4, 2^2)$ & $2$          \\
${}^2A_3(3)$ & $\cong$ & $\PSU(4, 3^2)$ & $3^2$        \\
${}^2A_5(2)$ & $\cong$ & $\PSU(6, 2^2)$ & $2^2$        \\ \hline
\end{tabular}
\ \ \ 
\begin{tabular}{|lcl|l|} \hline
$S$          &         &                   & $e$   \\ \hline
$B_2(2)$     & $\cong$ & $\PSp(4, 2)$      & $2$   \\
${}^2B_2(2)$ & $\cong$ & $\Sz(8)$          & $2^2$ \\
$B_3(2)$     & $\cong$ & $\PSp(6, 2)$      & $2$   \\
$B_3(3)$     & $\cong$ & $\POmega(7, 3)$   & $3$   \\
$C_3(2)$     & $\cong$ & $\PSp(6, 2)$      & $2$   \\
$D_4(2)$     & $\cong$ & $\POmega^+(8, 2)$ & $2^2$ \\
$G_2(3)$     &         &                   & $3$   \\
$G_2(4)$     &         &                   & $2$   \\
$F_4(2)$     &         &                   & $2$   \\
${}^2E_6(2)$ &         &                   & $2^2$ \\ \hline
\end{tabular}
\caption{Simple groups with exceptional Schur multiplier}
\label{table:excmult}
\end{center}
\end{table}

The simple groups $S$ with exceptional Schur multiplier $e$
are listed in Table~\ref{table:excmult}, cf. Table~6 of~\cite{CCN85}.
\end{proof}

Next we deal with alternating groups.

\begin{lem}\label{alternating}
Let $G$ be a perfect central extension of the alternating group $A_n$
on $n$ points, for $n \geq 7$.
Then $G$ has a diameter three subgroup.
\end{lem}

\begin{proof}
The claim holds for $n = 7$, by Lemma~\ref{quasisimple_compute},
and the alternating group on $n > 7$ points contains $A_7$.
\end{proof}

It remains to deal with the simple groups of Lie type.
The key result is about $\PSL(2, q)$.

\begin{lem}\label{SL2_lemma}
Let $q$ be a prime power with $q>3$.
If $5 \neq q \neq 9$ then $G:= \SL(2,q)$ has a diameter three subgroup.
\end{lem}

\begin{proof}
If $q$ is even then $G$ is simple.
Thus $G$ has a diameter three subgroup of order 2,
by Example~\ref{simple_group}.
Hence we may assume that $q$ is odd.
Let $H$ be a subgroup of order $3$ in $G$.
We will show that $\Gamma(G, H)$ has diameter $3$,
using Proposition~\ref{diam3_equiv}.
By Corollary~\ref{cond_ii_satisfied},
it suffices to verify that $H$ is rich in $G$.

%
%
%

In the following, we use the notation from~\cite[\S~38]{Dornhoff}.
We distinguish several cases.

\textit{Case 1:} $q \equiv 1 \pmod{3}$; in particular, $q \geq 7$.

The subgroup
$\langle a \rangle$ of order $q-1$ in $G$ contains an element $h:= a^l$ of
order $3$, and we may assume that $H = \langle h \rangle$. We need to show that,
for $\chi \in \Irr(G)$, we have $0 \neq [\chi_H,1_H]$, i.~e.,
\[
   \sigma_\chi:= \chi(1) + \chi(h) + \overline{\chi(h)} > 0.
\]
This is trivial if $\chi = 1_G$ or $\chi(h) = 0$.
If $\chi(1) \in \{q, (q+1)/2\}$ then $|\chi(h)| = 1$,
and thus $\sigma_\chi \ge 4-1-1>0$.
If $\chi(1) = q+1$ then $|\chi(h)| \leq 2$,
and thus $\sigma_\chi \geq 8-2-2 > 0$,
and the result is proved in this case.

\textit{Case 2:} $q \equiv 2 \pmod{3}$; in particular, $q \geq 11$.

In this case we may assume that $H = \langle h \rangle$ is contained in the
subgroup $\langle b \rangle$ of order $q+1$ in $G$. With notation as above,
we need to show that $\sigma_\chi > 0$,
and we may again assume that $\chi \neq 1_G$ and $\chi(h) \neq 0$.
If $\chi(1) \in \{ q, (q-1)/2 \}$ then $|\chi(h)| = 1$,
and thus $\sigma_\chi \geq 5-1-1>0$.
If $\chi(1) = q-1$ then $|\chi(h)| \leq 2$,
and thus $\sigma_\chi \geq 10-2-2>0$.
The result follows in this case as well.

\textit{Case 3:} $q \equiv 0 \pmod{3}$;
in particular, $q$ is a power of $3$, and $q \geq 27$.

In this case we may take $H = \langle h \rangle$ where $h:= c$. Again we need
to show that $\sigma_\chi > 0$, and we may assume that $\chi \neq 1_G$ and
$\chi(h) \neq 0$. If $\chi(1) = q \pm 1$ then $|\chi(h)| = 1$, and thus
$\sigma_\chi \geq 26-1-1>0$. If $\chi(1) = (q \pm 1)/2$ then $|\chi(h)| \leq
(1+\sqrt{q})/2$, and thus $\sigma_\chi \geq (q-1)/2 - 1 - \sqrt{q} > 0$.
This completes the proof of the lemma.
\end{proof}

\begin{lem}\label{typeA}
Let $G$ be a quasisimple finite group with $G/Z(G) \cong \PSL(n,q)$,
for some positive integer $n$ and some prime power $q$.
Then $G$ has a diameter three subgroup,
except when $G$ is isomorphic to one of the groups
$\SL(2,5) = 2.A_5$, $\SL(2,9) = 2.A_6$ or $6.A_6$.
\end{lem}

\begin{proof}
Since $G$ is quasisimple, we have $n \geq 2$ and $(2,2) \neq (n,q) \neq (2,3)$.
Moreover, $G$ is isomorphic to a factor group of the Schur cover $X$ of
$S:= \PSL(n,q)$,
and $Z(X)$ is isomorphic to the Schur multiplier
$M(S):= \mathrm{H}^2(S, \mathbb{C}^\times)$.
The order of $M(S)$ is $\gcd(n,q-1)$,
with the exceptions
\[
  (n,q) \in \{(2,4),\;(2,9),\;(3,2),\;(3,4),\;(4,2)\},
\]
which have been dealt with in Lemma~\ref{quasisimple_compute}.
Thus from now on we may assume that $(n, q)$ is not one of these exceptional
values.
Then $X$ is isomorphic to $\SL(n, q)$.
By Lemma~\ref{SL2_lemma}, we may assume that $n \geq 3$,
and that $Z(G) \neq 1$.
Recall that $\SL(n, q)$ contains subgroups isomorphic to $\SL(n-1, q)$
and subgroups isomorphic to $\SL(n, p)$ where $p$ is the prime dividing $q$.
This implies the lemma for $n = 3$.
(Note that $M(\SL(3, 3))$ and $M(\SL(3, 5))$ are trivial.)
The result for $n \geq 4$ then follows by induction on $n$.
\end{proof}

\begin{lem}\label{typeC}
Let $G$ be a quasisimple finite group with $G/Z(G) \cong \PSp(2n, q)$,
for some integer $n \geq 2$ and some prime power $q$.
Then $G$ has a diameter three subgroup.
\end{lem}

\begin{proof}
The exceptional case $(n, q) = (3, 2)$ has been done in
Lemma~\ref{quasisimple_compute}, we assume that $(n, q) \neq (3, 2)$.
Then the Schur multiplier of $\PSp(2n, q)$ is trivial when $q$ is even,
and of order $2$ when $q$ is odd.
Thus we may assume that $q$ is odd.
Then $\Sp(2n, q)$ is the Schur cover of $\PSp(2n, q)$.
Since our result is known for simple groups,
we may assume that $G \cong \Sp(2n, q)$ where $n > 1$ and $q$ is odd.
Recall that $\Sp(2n, q)$ contains subgroups isomorphic to $\Sp(2n-2,q)$.
Thus it suffices to prove the result for $n = 2$.
Since $\Sp(4, q)$ contains subgroups isomorphic to
$\Sp(2,q^2) \cong \SL(2, q^2)$, it suffices to consider the case $q = 3$.
But then we know that $G$ contains subgroups isomorphic to $\PSL(3, 3)$.
\end{proof}

\begin{lem}\label{type2A}
Let $G$ be a quasisimple finite group with $G/Z(G) \cong \PSU(n, q^2)$
for some integer $n \geq 3$ and some prime power $q$.
Then $G$ has a diameter three subgroup.
\end{lem}

\begin{proof}
First we consider the case $n = 3$.
Then $ q > 2$ since $\PSU(3, 2^2)$ is solvable.
Since the Schur multiplier of $\PSU(3, q^2)$ has order $\gcd(3, q+1)$,
we may assume that $q \equiv 2 \pmod{3}$
since otherwise $G$ is simple and thus we know already that the result holds.
Then $G$ is isomorphic to $\SU(3, q^2)$.
Since $\SU(3, q^2)$ contains subgroups isomorphic to
$\SU(2, q^2) \cong \SL(2, q)$,
Lemma~\ref{SL2_lemma} implies the lemma unless perhaps $q = 5$.
But $\SU(3, 5^2)$ contains subgroups isomorphic to $A_7$.

Next we consider the case $n = 4$.
The groups $\PSU(4, 2^2) \cong \PSp(4, 3)$ and $\PSU(4, 3^2)$ have been
dealt with in Lemma~\ref{quasisimple_compute}.
Thus we may assume that $q > 3$.
Then $\SU(4, q^2)$ is a Schur cover of $\PSU(4, q^2)$.
Since $\SU(4, q^2)$ contains subgroups isomorphic to $\SU(3, q^2)$,
the result follows from the previous case.

Finally, we consider the case $n \geq 5$.
The group $\PSU(6,2^2)$ has been dealt with in Lemma~\ref{quasisimple_compute}.
Thus we may assume that $(n,q) \neq (6,2)$.
Then $\SU(n, q^2)$ is a Schur cover of $\PSU(n, q^2)$,
and $\SU(n, q^2)$ has subgroups isomorphic to $\SU(n-1, q^2)$.
Thus our result follows by induction on $n$.
\end{proof}

\begin{lem}\label{typeOrth}
Let $G$ be a quasisimple finite group with $G/Z(G)$ of one of the
following types:
$\POmega(2n+1, q)$ ($n \geq 3$),
$\POmega^+(2n, q)$ ($n \geq 4$), or
$\POmega^-(2n, q)$ ($n \geq 4$), where $q$ is some prime power.
Then $G$ has a diameter three subgroup.
\end{lem}

\begin{proof}
In all three cases,
Table~\ref{table:Dynkin} lists a simple group $T$ such that $S = G/Z(G)$
has a subgroup that is isomorphic to a perfect central extension of $T$;
this subgroup is obtained by removing suitable nodes
from the Dynkin diagram of $S$.
Hence the claim follows from Lemma~\ref{typeA}.
\end{proof}

\begin{table}[htpb]
\begin{center}
\begin{tabular}{|lcl|l|lcl|} \hline
$S$          &         &                    &            & $T$ & & \\ \hline
$B_n(q)$     & $\cong$ & $\POmega(2n+1, q)$ & $n \geq 3$ & $A_{n-1}(q)$ & $\cong$ & $\PSL(n, q)$ \\
$D_n(q)$     & $\cong$ & $\POmega^+(2n, q)$ & $n \geq 4$ & $A_{n-1}(q)$ & $\cong$ & $\PSL(n, q)$ \\
${}^2D_n(q)$ & $\cong$ & $\POmega^-(2n, q)$ & $n \geq 4$ & $A_{n-2}(q)$ & $\cong$ & $\PSL(n-1, q)$ \\ \hline
$E_6(q)$     &         &                    &            & $A_5(q)$ & $\cong$ & $\PSL(6, q)$ \\
${}^2E_6(q)$ &         &                    &            & ${}^2A_5(q)$ & $\cong$ & $\PSU(6, q^2)$ \\
$E_7(q)$     &         &                    &            & $E_6(q)$ & & \\ \hline
\end{tabular}
\caption{Subgroups obtained from Dynkin diagrams}
\label{table:Dynkin}
\end{center}
\end{table}

\begin{proof} (of Theorem~\ref{thm_quasisimple})
According to the classification of the nonabelian finite simple groups,
each such group is either alternating (see Lemma~\ref{alternating}),
sporadic simple (see Lemma~\ref{quasisimple_compute}),
or a group of Lie type of the type $A$ (see Lemma~\ref{SL2_lemma} and
Lemma~\ref{typeA}), $C$ (see Lemma~\ref{typeC}),
${}^2A$ (see Lemma~\ref{type2A}),
$B$, $D$, ${}^2D$ (see Lemma~\ref{typeOrth}),
or of exceptional type.

Thus it remains to show the claim for the latter groups.
Table~\ref{table:Dynkin} (cf.~\cite[Table 5]{CCN85})
lists the series of those simple groups $S$ with nontrivial Schur multiplier
(omitting the cases that were dealt with in
Lemma~\ref{quasisimple_compute})
and a simple factor $T$ of a quasisimple subgroup of $S$,
where we know already
that any perfect central extension of $T$ has a diameter three subgroup.
The groups $T$ can be read off from the Dynkin diagrams of the groups $S$.

(The Schur multipliers of ${}^2B_2(q)$,
${}^3D_4(q)$, $G_2(q)$, ${}^2G_2(q)$, $F_4(q)$, ${}^2F_4(q)$,
and $E_8(q)$ are trivial, apart from the exceptions listed in
Lemma~\ref{quasisimple_compute}.)
\end{proof}

\section{Minimal groups with diameter three subgroups}%
\label{sect:minimal}

By Proposition~\ref{reduction},
we can ask for the smallest subgroups of a given group
that have a diameter three subgroup.
In this section, we study groups that are minimal w.~r.~t.~inclusion
in this respect.

Example~\ref{expl_frob_sub} states that $\AGL(1, p^2)$ has a unique
subgroup that is minimal in this sense, whose index in $\AGL(1, p^2)$
is the odd part of $p-1$.
Proposition~\ref{simplegroup} states that simple groups are never minimal.

Example~\ref{expl_frob_sub} also implies that the number of prime divisors
of the order of minimal groups that contain a diameter three subgroup
is not bounded.
Namely, for any natural number $t$ we may take the product $c$ of $t$
pairwise different primes, and choose a prime $p = k c -1$,
for some natural number $k$.
(Infinitely many such primes exist by Dirichlet's theorem.)
Then Example~\ref{expl_frob_sub} yields a group $G$ of order $p^2 \cdot d$,
where $d$ is a multiple of $c$, such that $G$ contains a
diameter three subgroup and is minimal with this property.

Note that the order of minimal non-nilpotent groups is divisible
by exactly two different primes.
Minimal non-supersolvable groups are solvable,
and Proposition~\ref{subsect:minSL2p} below shows that
minimal groups that have a diameter three subgroup need not be solvable.
See~\cite{Doerk} for properties of minimal non-nilpotent and
minimal non-supersolvable groups.

\begin{rem}
Proposition~\ref{lem:factor} states that factoring out certain normal
subgroups of a group with a diameter three subgroup yields again groups
with a diameter three subgroup.
Thus we could define minimality by going down to subgroups and by factoring
out normal subgroups if possible.
However, this would yield strange results.
For example, we will see in Proposition~\ref{SL2pminimal}
that the group $\SL(2, 7)$ is minimal w.r.t.~inclusion;
if we allow to take factors then we get $\PSL(2, 7)$,
which is \emph{not} minimal because its $A_4$ type subgroups
have diameter three subgroups.
\end{rem}

\subsection{Series of minimal examples}\label{subsect:minseries}

\begin{prop}\label{prop:series_minimal}
Let $G$ be a semidirect product of an elementary abelian $p$-group $E$
and a cyclic group $C$ of order $q$, a power of a prime $l \not= p$.
Assume that $C$ acts irreducibly on $E$
and that $G$ is minimal with the property that it has a diameter three
subgroup.

Let $G_i$ be a subgroup of index $q$ in the direct product
$G \times C_i$, where $C_i$ is a cyclic group of order $q \cdot l^i$,
such that the Sylow $l$-subgroup of $G_i$ is cyclic
and acts irreducibly on $E$.
Then $G_i$ has a diameter three subgroup and is minimal with this property.
\end{prop}

\begin{proof}
The group $G_i$ embeds into $G \times C_i$ by enlarging the centre,
so $G$ and $G_i$ are isoclinic for all $i$,
see the proof of Proposition~\ref{prop:isoclinic}.
Hence $G_i$ has a diameter three subgroup
by Lemma~\ref{lem:extendcentre}.

In order to show the minimality of $G_i$,
we show that no maximal subgroup of $G_i$ has a diameter three subgroup.
Let $M$ be a maximal subgroup of $G_i$.
Since $G_i$ is solvable, $[G_i:M]$ is a prime power.

If this prime is $l$ then $M$ contains $E$,
thus $M$ has index $l$ in $G_i$, and $M$ can be embedded into
$U \times C_{i-1}$, where $U$ is the subgroup of index $l$ in $G$.
Thus $M$ is isoclinic with $U$ and hence has no diameter three subgroup,
by the minimality of $G$.

If this prime is $p$ then $M$ is a Sylow $l$-subgroup of $G_i$,
by the irreducibility of the action on $E$,
thus $M$ has no diameter three subgroup.
\end{proof}

\begin{example}\label{series_minimal}
Applying Proposition~\ref{prop:series_minimal} to 
the Frobenius group $G$ of order $2^n (2^n - 1)$
from Example~\ref{expl_frob}, where $p = 2^n - 1$ is a prime,
yields minimal examples $2^2:3, 2^2:9, 2^2:27, \ldots$,
$2^3:7, 2^3:49, 2^3:343, \ldots$,
$2^5:31, 2^5:961, \ldots$,
$2^7:127, 2^7:16129$.

Starting from $G = 3^2:8$, we get $3^2:16, 3^2:32, \ldots$,
and $G = 2^4:5$ yields $2^4:25, 2^4:25, \ldots$.
See Table~\ref{table:isoclasses} for more examples.
\end{example}

\subsection{The groups $\SL(2, p)$}\label{subsect:minSL2p}

\begin{prop}\label{SL2pminimal}
Let $p$ be a prime.
Then $\SL(2, p)$ is a minimal group (w.~r.~t.~inclusion)
that contains a diameter three subgroup if and only if $p > 5$.
\end{prop}

\begin{proof}
The group $\SL(2, 2) \cong S_3$ is supersolvable,
and $\SL(2, 3)$ and $\SL(2, 5)$ are Frobenius complements,
thus we know that these groups do not have diameter three subgroups.

Now assume $p > 5$.
Lemma~\ref{SL2_lemma} shows that $G = \SL(2, p)$ has
a diameter three subgroup.
We show that the proper subgroups of $G$ do not have such a subgroup.
Let $\pi\colon G \rightarrow G/Z(G) \cong \PSL(2, p)$ be the natural
epimorphism.
The subgroups of the simple group $\pi(G)$
are listed in \cite[Satz~II.8.27]{HuppertI},
there are eight types of subgroups.
The types (1)--(3) and (7) are cyclic or metacyclic,
in particular supersolvable,
hence also their preimages under $\pi$ are supersolvable
and thus do not have a diameter three subgroup,
by Corollary~\ref{cor:supersolvable}.
The preimages of subgroups of the types (4)--(6) are isomorphic to
$\SL(2, 3)$, $\SL(2, 5)$, or the unique group of order $48$
with exactly one involution s.~t. the factor modulo the centre
is isomorphic to $S_4$;
the former two groups have been dealt with above,
and if the latter had a diameter three subgroup then
it would be of order divisible by $3$, which would imply a
diameter three subgroup of order $3$ in $S_4$,
by Proposition~\ref{lem:factor},
which contradicts Corollary~\ref{cor_easy}.
Finally, subgroups of the type (8) do not occur because $p$ is prime.
\end{proof}

\subsection{Small groups}\label{subsect:minsmall}

Using the library of small groups~\cite{SmallGroup},
we computed the groups $G$ of order up to $2\,000$,
up to isomorphism, that contain nontrivial rich subgroups $H$.
There are exactly $52\,239$ such isomorphism types,
$33\,523$ of them have order $1\,536$.

If we consider only those groups $G$ such that no proper subgroup of $G$
has this property,
we get exactly $163$ such groups, up to isomorphism.
These groups lie in $40$ isoclinism classes.

(Fortunately, enough information is available such that one need not
really run over all isomorphism types of groups.
For example, most of the $408\,641\,062$ groups of order $1\,536$
are supersolvable and hence need not be checked.
In the end, no group of this order turned out to be minimal.)

Table~\ref{table:isoclasses} shows one representative of each
isoclinism class.
The first column contains the number $k$ of isomorphism types of those groups
in the isoclinism class that are among the $163$ minimal examples.
The second and third column list the values $|G|$ and $i$ such that the
group $G$ is the $i$-th group of its order, according to the numbering
in~\cite{SmallGroup}; the values of Frobenius groups are shown in boldface.

Examples of isoclinic minimal examples are described in
Section~\ref{subsect:minseries}.

Diameter three subgroups $H$ in groups of order at most $2\,000$
have order at most $16$.

There are exactly $11$ isomorphism types of groups of order at most $2\,000$
which have nontrivial rich subgroups that are not diameter three subgroups.
Among these groups are the Frobenius groups $3^3\spl 13$ and $2^6\spl 21$
of the orders $351$ and $1\,344$, respectively,
and direct products $A \times 3^3\spl 13$ where $A$ is an abelian group
of order at most $5$.

\begin{table}[htpb]
\small
\[
   \begin{array}{|r|r|r|l|} \hline
      k & |G| & i & G \\
    \hline
      5 & \textbf{12} & \textbf{3} & 2^2\spl 3 \cong A_4 \\
      2 & \textbf{56} & \textbf{11} & 2^3\spl 7 \\
      5 & \textbf{72} & \textbf{39} & 3^2\spl 8 \\
      9 & \textbf{72} & \textbf{41} & 3^2\spl Q_8 \\
      3 & \textbf{80} & \textbf{49} & 2^4\spl 5 \\
      2 & 160 & 199 & 2^{1+4}_-\spl 5 \\
      4 & 216 & 86 & 3^{1+2}_+\spl 8 \\
      6 & 216 & 88 & 3^{1+2}_+\spl Q_8 \\
     18 & 288 & 393 & 3^2\spl (8\spl 4) \\
      1 & 336 & 114 & \SL(2,7) \\
      1 & \textbf{351} & \textbf{12} & 3^3\spl 13 \\
      6 & 576 & 1966 & 3^2\spl (16\spl 4) \\
      9 & 576 & 1967 & 3^2\spl (16\spl 4) \\
      5 & 576 & 1973 & 3^2\spl ((8 \times 2)\spl 4) \\
      9 & 576 & 1976 & 3^2\spl ((8\spl 2)\spl 4) \\
      2 & \textbf{600} & \textbf{148} & 5^2\spl (3\spl 8) \\
      3 & \textbf{600} & \textbf{149} & 5^2\spl 24 \\
      2 & \textbf{600} & \textbf{150} & 5^2\spl \SL(2,3) \\
      2 & 648 & 641 & 3^3\spl \SL(2,3) \\
      2 & \textbf{784} & \textbf{160} & 7^2\spl 16 \\
    \hline
   \end{array}
   \ \ \ 
   \begin{array}{|r|r|r|l|} \hline
      k & |G| & i & G \\
    \hline
      3 & \textbf{784} & \textbf{162} & 7^2\spl Q_{16} \\
      9 & 864 & 676 & 3^{1+2}_+\spl (8\spl 4) \\
      1 & \textbf{992} & \textbf{194} & 2^5\spl 31 \\
     10 & 1152 & 4900 & 3^2\spl (((4 \times 2)\spl 4)\spl 4) \\
     10 & 1152 & 5070 & 3^2\spl ((2 \times (4\spl 4))\spl 4) \\
      6 & 1152 & 5232 & 3^2\spl ((2 \times ((4 \times 2)\spl 2))\spl 4) \\
      4 & 1152 & 6492 & 3^2\spl ((8\spl 4)\spl 4) \\
      3 & 1152 & 6577 & 3^2\spl (((2^3)\spl 4)\spl 4) \\
      3 & 1152 & 6619 & 3^2\spl (((2^3)\spl 4)\spl 4) \\
      3 & 1152 & 7054 & 3^2\spl (32\spl 4) \\
      1 & 1152 & 7092 & 3^2\spl (32\spl 4) \\
      1 & 1320 & 13 & \SL(2,11) \\
      1 & \textbf{1620} & \textbf{419} & 3^4\spl (5\spl 4) \\
      1 & \textbf{1620} & \textbf{420} & 3^4\spl 20 \\
      1 & 1620 & 421 & 3^4\spl (5\spl 4) \\
      1 & 1728 & 2787 & 3^{1+2}_+\spl (16\spl 4) \\
      3 & 1728 & 2788 & 3^{1+2}_+\spl (16\spl 4) \\
      2 & 1728 & 2794 & 3^{1+2}_+\spl (4^2\spl 4) \\
      3 & 1728 & 2797 & 3^{1+2}_+\spl ((8\spl 2)\spl 4) \\
      1 & 1800 & 270 & 5^2\spl (9\spl 8) \\
    \hline
   \end{array}
\]
\caption{Small groups that are minimal with a nontrivial rich subgroup}
\label{table:isoclasses}
\end{table}

\section{Diameter and Depth}\label{sect:depth}

A notion of depth can be defined for subrings of a ring (cf.~\cite{BKK}).
Here we are only interested in complex group algebras of finite groups and
their subgroups. In this situation the depth can be computed in terms of the
Frobenius matrix. More precisely, let $H$ be a proper subgroup of a finite
group $G$, and set $M := F(G, H)$.
Then we have
\[
   S := MM^\top = ([\varphi^G,\psi^G])_{\varphi,\psi \in \Irr(H)}.
\]
For a positive integer $m$, one has that $H$ is of \emph{depth} $n=2m+1$ in $G$
if and only if $S^{m+1} \leq qS^m$ for some $q>0$, and that $H$ is of
\emph{depth} $n=2m$ in $G$ if and only if $S^mM \leq qS^{m-1}M$ for some $q>0$.
Here the inequality $A \leq B$ between real matrices $A = (a_{ij})$ and $B =
(b_{ij})$ of the same format is defined by $a_{ij} \leq b_{ij}$ for all $i,j$.

It is known that depth $n$ always implies depth $n+1$.
Thus the \emph{minimal depth} $d(H, G)$ is of particular interest.
It is also known that $H$ is of depth $2$ in $G$ if and only if $H$ is
a normal subgroup of $G$.
Several papers have investigated subgroups of depth $3$ (cf.~\cite{B,BK}).
Here we point out connections to the diameter of $\Gamma(G, H)$.

\begin{prop}
Let $H$ be a nontrivial core-free subgroup of a finite group $G$.
Then the following assertions hold:

(i) If $\Gamma(G, H)$ has diameter $3$ then $[\varphi^G, \psi^G] > 0$
    for all $\varphi, \psi \in \Irr(H)$.

(ii) If $[\varphi^G, \psi^G] > 0$ for all $\varphi, \psi \in \Irr(H)$
     then $H$ is of depth $3$ in $G$.

(iii) If $H$ is of depth $3$ in $G$ then the diameter of $\Gamma(G, H)$
      is $3$ or $4$.
\end{prop}

\begin{proof}
(i) is an immediate consequence of condition (b)~(ii) in
 Proposition~\ref{diam3_equiv}.

(ii) follows from the above definition of depth.

(iii) Suppose that $H$ is of depth $3$ in $G$.
 Since $\Gamma(G, H)$ is connected, \cite[Theorem~3.6]{BKK} implies
 that the distance between any two distinct $\varphi, \psi \in \Irr(H)$ is $2$.
 Thus the distance between any $\varphi \in \Irr(H)$ and any $\chi \in \Irr(G)$
 is at most $3$,
 and the distance between any two distinct $\chi, \eta \in \Irr(G)$
 is at most $4$.
 Thus $\Gamma(G, H)$ has diameter $3$ or $4$.
\end{proof}

Thus diameter $3$ implies depth 3,
and conversely depth 3 implies diameter $3$ or $4$.
Note that it can happen that the depth is $3$ whereas the diameter is $4$,
as the example $S_2 < S_3$ shows.
Note also that the depth is defined also for subgroups with nontrivial core;
for example, the depth of the Sylow $2$-subgroup $H$ in the dihedral group
$G$ of order $12$ is $3$, but $\Gamma(G, H)$ consists of two connected
components, each a path of length $4$.



\begin{center}
 {\bf Acknowledgements}
\end{center}

\noindent
Work on this paper started when the fourth author visited Budapest;
he is grateful for all the hospitality and inspiration received there.
The first author gratefully acknowledges support by the German Research
Foundation (DFG)
-- Project-ID 286237555 --
within the SFB-TRR 195 {\it Symbolic Tools in Mathematics
and their Applications}.
The research in this paper was also supported by the NKFI-Grant
No.~138596.

\bibliographystyle{plain}
\bibliography{diam3}

\begin{thebibliography}{10}

\bibitem{BW}
M.~J.~J. Barry and M.~B. Ward.
\newblock Simple groups contain minimal simple groups.
\newblock {\em Publ. Mat.}, 41(2):411--415, 1997.

\bibitem{SmallGroup}
H.~U. Besche, B.~Eick, E.~O'Brien, and M.~Horn.
\newblock {SmallGrp}, the \textsf{GAP} small groups library, {V}ersion 1.5.3.
\newblock \verb!https://gap-packages.github.io/smallgrp/!, May 2023.
\newblock \textsf{GAP} package.

\bibitem{CTblLib}
T.~Breuer.
\newblock {CTblLib}, the \textsf{GAP} character table library, {V}ersion 1.3.9.
\newblock \verb!https://www.math.rwth-aachen.de/~Thomas.Breuer/ctbllib!, Mar
  2024.
\newblock \textsf{GAP} package.

\bibitem{B}
S.~Burciu.
\newblock Subgroups of odd depth---a necessary condition.
\newblock {\em Czechoslovak Math. J.}, 63(138)(4):1039--1048, 2013.

\bibitem{BK}
S.~Burciu and L.~Kadison.
\newblock Subgroups of depth three.
\newblock In {\em Surveys in differential geometry. {V}olume {XV}.
  {P}erspectives in mathematics and physics}, volume~15 of {\em Surv. Differ.
  Geom.}, pages 17--36. Int. Press, Somerville, MA, 2011.

\bibitem{BKK}
S.~Burciu, L.~Kadison, and B.~K\"ulshammer.
\newblock On subgroup depth.
\newblock {\em Int. Electron. J. Algebra}, 9:133--166, 2011.
\newblock With an appendix by S. Danz and B. K\"ulshammer.

\bibitem{CCN85}
J.~H. Conway, R.~T. Curtis, S.~P. Norton, R.~A. Parker, and R.~A. Wilson.
\newblock {\em {$\mathbb{ATLAS}$} of finite groups}.
\newblock Oxford University Press, Eynsham, 1985.
\newblock Maximal subgroups and ordinary characters for simple groups, With
  computational assistance from J. G. Thackray.

\bibitem{Doerk}
K.~Doerk.
\newblock Minimal nicht \"uberaufl\"osbare, endliche {G}ruppen.
\newblock {\em Math. Z.}, 91:198--205, 1966.

\bibitem{Dornhoff}
L.~Dornhoff.
\newblock {\em Group representation theory. {P}art {A}: {O}rdinary
  representation theory}, volume~7 of {\em Pure and Applied Mathematics}.
\newblock Marcel Dekker, Inc., New York, 1971.

\bibitem{Feit}
W.~Feit.
\newblock {\em Characters of finite groups}.
\newblock W. A. Benjamin, Inc., New York-Amsterdam, 1967.

\bibitem{Hekster}
N.~S. Hekster.
\newblock On the structure of {$n$}-isoclinism classes of groups.
\newblock {\em J. Pure Appl. Algebra}, 40(1):63--85, 1986.

\bibitem{HuppertI}
B.~Huppert.
\newblock {\em Endliche {G}ruppen. {I}}, volume 134 of {\em Grundlehren der
  mathematischen Wissenschaften}.
\newblock Springer-Verlag, Berlin-New York, 1967.

\bibitem{HBIII}
B.~Huppert and N.~Blackburn.
\newblock {\em Finite groups. {III}}, volume 243 of {\em Grundlehren der
  mathematischen Wissenschaften}.
\newblock Springer-Verlag, Berlin-New York, 1982.

\bibitem{Isaacs}
I.~M. Isaacs.
\newblock Subgroups with the character restriction property.
\newblock {\em J. Algebra}, 100(2):403--420, 1986.

\bibitem{OSCAR}
The~{OSCAR} Team.
\newblock {OSCAR} -- {O}pen {S}ource {C}omputer {A}lgebra {R}esearch system,
  version 1.0.0.
\newblock \verb!https://www.oscar-system.org!, 2024.

\end{thebibliography}

\vspace*{2cm}

\noindent
T. Breuer,
Lehrstuhl f\"ur Algebra und Zahlentheorie,
RWTH Aachen University,
Pontdriesch 14-16,
D-52062 Aachen, Germany, \\
e-mail: \texttt{sam@math.rwth-aachen.de} \\
\\
L. H\'ethelyi,
Department of Algebra,
Budapest University of Technology and Economics,
H-1111 Budapest,
M\H uegyetem rkp. 3-9,
Hungary,\\
e-mail: \texttt{lhethelyi@gmail.com} \\
\\
B. K\"ulshammer,
Institut f\"ur Mathematik,
Friedrich-Schiller-Universit\"at,
D-07737 Jena, Germany, \\
e-mail: \texttt{kuelshammer@uni-jena.de} \\

\end{document}